\documentclass{aptpub}
\authornames{A.\ D.\ Barbour, K.\ Hamza, H.\ Kaspi \& F.\ C.\ Klebaner}
\shorttitle{Escape from the boundary}

\usepackage{amssymb}
\usepackage{color,amsmath,amsfonts}

\def\numberlikeadb{\global\def\theequation{\thesection.\arabic{equation}}}
\numberlikeadb

\newcommand{\halmos}{\rule{1ex}{1.4ex}}
\newcommand{\proofbox}{\hspace*{\fill}\mbox{$\halmos$}}

\newcommand{\pr}{{\mathbb P}}

\newcommand{\reals}{{R}}
\newcommand{\ints}{\mathbb{Z}}

\newcommand{\dtv}{d_{\mbox{\tiny TV}}}

\newcommand{\eqa}{\begin{eqnarray}}
\newcommand{\ena}{\end{eqnarray}}
\newcommand{\eq}{\begin{equation}}
\newcommand{\en}{\end{equation}}
\newcommand{\eqs}{\begin{eqnarray*}}
\newcommand{\ens}{\end{eqnarray*}}

\def\zz{\ints}

\def\a{\alpha}
\def\b{\beta}
\def\g{\gamma}

\def\d{\delta}
\def\D{\Delta}
\def\e{\varepsilon}
\def\h{\eta}

\def\th{\theta}

\def\k{\kappa}

\def\n{\nu}

\def\r{\rho}
\def\s{\sigma}

\def\t{\tau}
\def\f{\varphi}
\def\ch{\chi}
\def\ps{\psi}

\def\nin{\noindent}

\def\Blb{\left\{}
\def\Brb{\right\}}

\def\giv{\,|\,}
\def\Giv{\,\Big|\,}
\def\non{\nonumber}

\def\Eq{\ =\ }

\def\Le{\ \le\ }

\def\Ref#1{{\rm (\ref{#1})}}

\def\bp{\begin{proof}}
\def\ep{\end{proof}}

\def\bone{{\bf 1}}

\def\Bl{\left(}
\def\Br{\right)}

\def\half{{\textstyle{\frac12}}}

\def\ff{{\cal F}}

\def\ignore#1{}
\def\ex{{\mathbb E}}

\def\Z{\ints}

\def\R{\reals}

\def\jj{{\mathcal J}}

\def\sjj{\sum_{J\in \jj}}

\def\uii{^{(i)}}

\def\uN{^{(N)}}
\def\uNt{^{(N,\th)}}

\def\tx{{\tilde x}}

\def\ui{^{(1)}}
\def\uii{^{(i)}}
\def\ut{^{(2)}}
\def\uh{^{(3)}}

\def\Blm{\left|}
\def\Brm{\right|}

\def\JJ{{\mathcal J}}
\def\tm{{\tilde m}}

\def\reals{\mathbb{R}}
\def\tfrac#1#2{{\textstyle{\frac{#1}{#2}}}}
\def\tc{{\tilde c}}
\def\bg{{\bar g}}
\def\sjjt{\sum_{J \in \JJ_2}}
\def\zz{{\cal Z}}
\def\Exp{{\rm Exp\,}}
\def\tG{{\widetilde H}}
\def\hG{{\widehat H}}
\def\txi{{\tilde\xi}}
\def\bu{{\mathbf{u}}}
\def\bv{{\mathbf{v}}}
\def\tod{\to_d}
\def\hxi{{\hat\xi}}

\def\hE{{\widehat E}}
\def\uf{{^{(4)}}}
\def\tD{{\widetilde \D}}
\def\uN{^{(N)}}
\def\uv{^{(5)}}
\def\law{{\cal L}}
\def\etal{{\it et al.\/}}

\def\Def{\ :=\ }
\def\cst{\g}
\def\tk{{\tilde k}}
\def\bk{{\bar k}}
\def\hk{{\hat k}}
\def\KKK{{K}}
\def\hhh{{\h}}
\def\hht{{\hat t}}

\def\bv{\mathbf{v}}
\def\bu{\mathbf{u}}

\def\nn{{\cal N}}
\def\kk{{\cal K}}
\def\kkkc{K_c}
\def\CCC{F}
\def\bP{{\mathbf P}}
\def\DDD{\f}

\begin{document}

\title{Escape from the boundary in Markov population processes}

\authorone[Universit\"at Z\"urich]{A. D. Barbour}%
\authortwo[Monash University]{K. Hamza}%
\authorthree[Technion, Haifa]{Haya Kaspi}
\authortwo[Monash University]{F. C. Klebaner}
\addressone{Institut f\"ur Mathematik, Universit\"at Z\"urich, Winterthurertrasse 190, CH-8057 Z\"urich}
\addresstwo{School of Mathematical Sciences, Monash University, Clayton, VIC 3800, Australia}
\addressthree{Faculty of Industrial Engineering and Management, Technion, Haifa 32000, Israel}

\begin{abstract}
Density dependent Markov population processes in large populations
of size~$N$ were shown by Kurtz (1970, 1971) to be well approximated
over finite time intervals by the solution of the differential equations that
describe their average drift, and to exhibit stochastic fluctuations
about this deterministic solution on the scale~$\sqrt N$ that can be
approximated by a diffusion process.
Here, motivated by an example from evolutionary biology,
we are concerned with describing how such a process leaves an
absorbing boundary.  Initially, one or more of the
populations is of size much smaller than~$N$, and the length
of time taken until all populations have sizes comparable to~$N$
then becomes infinite as $N \to \infty$.
Under suitable assumptions, we show that in the early stages of development,
up to the time when all populations have sizes at least $N^{1-\a}$, for
$1/3 < \a < 1$, the process can be accurately approximated in total variation by a Markov
branching process.  Thereafter, it is well approximated by
the deterministic solution starting from the original initial point,
but with a random time delay.  Analogous behaviour is also established
for a Markov process approaching an equilibrium on a boundary, where
one or more of the populations become extinct.
\end{abstract}

 \noindent

\section{Introduction}\label{introduction}
\setcounter{equation}{0}
A continuous time version of the two morphs stage in
the bare bones evolution model of Klebaner \etal\ (2011, Section~3) can be represented as a
pure jump Markov process~$X_N$ on~$\Z_+^2$, with the first component the count
of wild type individuals, initially around their carrying capacity, and the second
the count of mutant individuals.  The transition rates are as follows:
\eq\label{ex-rates}\begin{array}{cclcl}
   X &\to& X + (1,0) \quad &\mbox{at rate}& \quad a_1X_1; \\
   X &\to& X + (-1,0) \quad &\mbox{at rate}& \quad X_1\{(X_1/N) + \g (X_2/N)\}; \\
   X &\to& X + (0,1) \quad &\mbox{at rate}& \quad a_2X_2; \\
   X &\to& X + (0,-1) \quad &\mbox{at rate}& \quad X_2\{\g (X_1/N) + (X_2/N)\}.
   \end{array}
\en
Initially, $X_1$ has a value near its carrying capacity~$Na_1$, and~$X_2 = 0$.  At some time, which
we call~$0$, $Z_0$ mutant individuals are introduced into the population;  $Z_0$
is thought of as fixed, irrespective of the (large) value of~$N$.
The mutants and wild type individuals differ only through their birth rates $a_1$ and~$a_2$.
Each species has {\it per capita\/} death rate given by the density
of its own population, together with an additional component of $\g$ times the density
of individuals of the other species.  If $\g > 1$, members of the other species result
in a higher mortality rate than if they were of the same species; if $\g < 1$, they
result in a lower mortality rate than if they were of the same species, favouring the
possibility of co-existence.  If $a_2 < \g a_1$,
the mutants have negligible chance of survival, but, if $a_2 > \g a_1$, there is
a non-zero probability $p_N(Z_0) \approx 1 - (\g a_1/a_2)^{Z_0}$ that the mutant
strain will become established.  In this case,
if also $a_1 > \g a_2$, the two populations will eventually come to co-exist;  if, instead,
$a_1 < \g a_2$, the wild type population will be driven to extinction. Note that, as
expected, co-existence is impossible if $\g > 1$.
In this paper, we are primarily interested
in describing how the process develops, up to the time at which the
mutants represent a positive fraction of the population, when~$N$ is large.  We also
examine the detail of how the wild type becomes extinct, when $a_1 < \g a_2$.

This process is a particular example of a more general family of processes,
that we now investigate.  We suppose that~$X_N$ is a Markov population
process on $\Z_+^d$ having transition rates
\eq\label{rates}
  X \ \to\ X + J \quad\mbox{at rate}\ Ng^J(N^{-1}X),\qquad X \in \Z_+^d,
\en
for $J \in \jj$, where~$\jj$ is a finite subset of~$\Z^d$.  We assume that the
functions~$g^J$ are continuously differentiable for $x\in \reals_+^d$, and we write
\eq\label{F-def}
  F(x) \Def \sjj Jg^J(x),\qquad x \in \reals_+^d,
\en
to denote the infinitesimal drift of the process~$x_N := N^{-1}X_N$.
Letting $\{P^J: J\in \mathcal{J}\}$ be independent rate 1 Poisson processes,
the evolution of~$x_N$ can be described (Kurtz, 1977) by the equation 
\eqa\label{stochastic}
   x_N(t) &=& x_N(0) + N^{-1}\sjj J P^J(NG_N^J(t)) \non\\
            &=& x_N(0) + \int_0^t F(x_N(u))\,du + m_N(t), \label{0-new}
\ena
where
\eq\label{AB-mN-def}
     m_N(t) \Def \sjj J\{N^{-1}P^J(NG_N^J(t)) - G_N^J(t)\},
\en
and $G_N^J(t) := \int_0^t g^J(x_N(u))du$.
$m_N$ is a well-behaved vector valued martingale.
In differential form, \Ref{stochastic} becomes
\eq\label{stochastic-de}
    dx_N(t) \Eq F(x_N(t)) + dm_N(t),
\en
and the corresponding `deterministic equations',  given by leaving out the martingale
innovations, are
\eq\label{deterministic}
    \frac{d\xi}{dt} \Eq  F(\xi).
\en

Our interest here is in deriving an approximation to the process~$x_N$ in circumstances in which
the initial state is close to~$\bar x$,
an unstable equilibrium point of the equations~\Ref{deterministic}, as in the bare bones
example given above.  In the seminal papers of Kendall~(1956) and Whittle~(1955), written in
the context of Bartlett's~(1949) Markovian SIR epidemic process, a basic description
was proposed.  Such processes should behave much like branching processes near~$\bar x$, as far as
those components in which numbers are small are concerned, and should then look more and more like
solutions to the deterministic equations as the numbers grow.  The deterministic part of the approximation
was established for general Markov population processes
in Kurtz~(1970, Theorem~(3.1)), who showed that, if $\lim_{N\to\infty}x_N(0) = x_0$, then
$\sup_{0\le t\le T}|x_N(t) - \xi(t)| \to 0$ in distribution, for any finite~$T > 0$,
where~$\xi$ satisfies~\Ref{deterministic} with~$\xi(0) = x_0$.  In particular, if~$x_0 = \bar x$,
Kurtz's~(1970) theorem implies that~$x_N(t)$ stays asymptotically close to~$\bar x$ over any fixed finite
time interval.  However, the deterministic solution~$\xi_N$ starting with $x_N(0)$ close to~$\bar x$
may still eventually escape from~$\bar x$, but the time that it takes to do so is
asymptotically infinite as $N\to\infty$, so that Kurtz's~(1970) theorem is not
suitable for describing what eventually happens.  Such outcomes may nonetheless be
of considerable practical importance in applications. The aim of this paper is to
show that the Kendall--Whittle description can indeed be established
in considerable generality, and to give some measure of the accuracy of the resulting
approximation.

Under appropriate conditions, we prove that the process~$x_N$, if it indeed escapes from~$x_0$,
then closely follows the path of the solution to the deterministic equations, but with a random
time shift, and that the time required to escape from~$x_0$ is of order~$O(\log N)$.
This behaviour is exactly what one might expect on the basis of the Kendall--Whittle
description, with the random time shift reflecting the essential randomness that occurs
in the early stages of the branching phase.  However, proving that it is actually the case is
not so easy.  A main difficulty is presented by the asymptotically infinite length of time that
elapses, while the process is escaping from the boundary, since this necessitates good control over the
behaviour of the process over long time intervals.  A related difficulty is to keep control
of the branching approximation for a long enough time to ensure that the subsequent development
is indeed almost deterministic.  Our approach is to establish extremely accurate approximation,
in terms of the total variation distance between the probability distributions of the
two processes, over a very long initial time interval.  Once this has been
achieved, the subsequent development can be well enough described by the
deterministic solution.
 We then go on to prove complementary results, describing the behaviour of a process that
approaches a stable equilibrium point of the deterministic equations at which
some coordinates of the process take the value zero.

\subsection{Assumptions}\label{assumptions}

Our general setting is as follows; the specialization to the bare bones example is given in
Section~\ref{example}.
Denote by~$x\ui$ the first~$d_1$ components
of~$x$ and by~$x\ut$ the remaining $d_2 = d - d_1$ components,
and split $J = (j_1,\ldots, j_{d}) = (J^{(1)},J^{(2)})$ in the same way.
\ignore{
 $J=(j_1,\ldots, j_{d}):=(J^{(1)},J^{(2)})$, $J^{(1)}=(j_1,\ldots, j_{d_1})$,
$J^{(2)}=(j_{d_1+1},\ldots, j_{d_2})$. Let $s(J)$ be one of the components of $J^{(2)}$.
}
For transitions with $J\ut \neq 0$, suppose that the rates are always of the form $g^J(x) = \bg^J(x)x_{s(J)}$,
for some $s(J)$ such that $d_1 < s(J) \le d$, and that $\bg^J(x_0) > 0$;
we also assume that $J_i \ge 0$
for all~$i\neq s(J)$ such that \hbox{$d_1 < i \le d$,} and that $J_{s(J)} \ge -1$.
We denote the set of all such transitions by~$\JJ_2$.  These assumptions are natural in a
population context; in particular, if the constraints on the elements of such~$J$ are
violated, some of the components could become negative. The function~$F$ can now
be written in the form
\eq\label{F-structure}
   F(x) \Eq \left(\begin{array}{c}
                   A(x) \\ B(x)
                  \end{array}\right)  x\ut + \left(\begin{array}{c}
                   c(x\ui) \\ 0
                  \end{array}\right),
\en
where, for each~$x$, $A(x)$ and~$B(x)$ are $d_1\times d_2$ and $d_2\times d_2$ matrices respectively,
and~$c(x\ui)$ is a $d_1$-vector.  Suppose that~$x_0\ui$ is a strongly stable equilibrium of
$d\xi\ui/dt = c(\xi\ui)$ and that $x_0\ut = 0$. Then the solution~$\xi$ of the deterministic
equations starting at~$x_0$ is the constant~$x_0$, and the stochastic system~$x_N$, if started near~$x_0$
with~$x_N\ut(0) = 0$, typically spends an amount of
time that is at least exponential in~$N$ before leaving the vicinity of~$x_0$
(Barbour \& Pollett~2012, Theorem~4.1).  However, if the initial
value $x_N\ut(0)$ is not~$0$, but takes the value $x_{N,0}\ut = N^{-1}Z_0$,
for some $0 \neq Z_0 \in \Z_+^{d_2}$, and if~$B_0 := B(x_0)$ is such that~$\xi_N$,
the solution of~\Ref{deterministic}
starting from this initial condition, leaves the neighbourhood of the boundary,
then~$x_N$ has positive probability of doing so as well.

We shall suppose henceforth that~$x_N(0) = x_{N,0}$ satisfying
$|x_{N,0}\ui - x_0\ui| \le N^{-5/12}$.
Under the equilibrium distribution for~$x_N\ui$ when~$x_N\ut = 0$, typical values of $|x_{N,0}\ui - x_0\ui|$
are of order~$O(N^{-1/2})$, so such a starting condition is reasonable.  Suppose also
that $x_{N,0}\ut = N^{-1}Z_0$.
Our assumptions imply that~$B$ has non-negative off-diagonal entries near~$x_0$;
we assume also that it is irreducible, and that the largest eigenvalue~$\b_0$ of~$B_0$ is positive.
In addition, the elements of the matrices $A$ and~$B$ are assumed to be
continuously differentiable functions of~$x$.
The stability of~$x_0\ui$ is expressed by assuming that the function~$c$ is of the form
\eq\label{c-assn}
   c(w) \Eq C(w - x_0\ui) + \tc(w),\quad w \in \reals_+^{d_1},
\en
where $C$ is a fixed $d_1\times d_1$ matrix such that, for some $\cst_1 < \infty$, 
\eq\label{C-bound}
   |e^{Ct}x| \Le \cst_1 |x|, \quad x \in \reals^{d_1},\ t\ge0,
\en
as is the case if all the eigenvalues of~$C$ have negative real part,
and where, for some $\kkkc ,\r_1 > 0$, and for $w_1,w_2 \in \R_+^{d_1}$
such that $\max_{i=1,2}|w_i-x_0\ui| \le \r_1$,
\eq\label{deriv-bnd}
   |\tc(w_1) - \tc(w_2)| \Le \kkkc |w_1-w_2|\{|w_1-w_2| + \min_{i=1,2}|w_i-x_0\ui|\}.
\en
{}From the Perron--Frobenius theorem, there also exist $0 <\cst_3 < \cst_2 < \infty$ such that
\eq\label{matrix-bnds}
    |e^{B_0t}x| \Le \cst_2 e^{\b_0 t}|x|, \quad x \in \reals^{d_2},\ t \ge 0,
\en
and
\eq\label{matrix-bnds-2}
   |e^{B_0t}x| \ \ge\ \cst_3 e^{\b_0t}|x|,\quad x \in \reals_+^{d_2},\ t \ge 0.
\en
We also choose~$0 < \r_2 \le \r_1$ small enough that
\eq\label{rho-2-cond}
  b_*^J\ :=\ \inf_{|x-x_0| \le \r_2}|\bg^J(x)|\ >\ 0 \
                    \mbox{for all}\ J \in \jj_2.
\en
We denote by~$\|G\|$ the matrix norm $\|G\| := \sup_{y\colon\,|y|=1}\{|Gy|\}$.  For matrix
functions~$G(x)$, we write $\|G\|_{\r} := \sup_{|x-x_0| \le \r}\|G(x)\|$, and 
$$
  \|DG\|_\r \Def \sup_{|x-x_0| \le \r,|x'-x_0| \le \r}\{|x-x'|^{-1}\|B(x) - B(x')\|\}. 
$$
In all the arguments that follow, constants involving the symbol~$k$ are defined solely
in terms of the functions $A$, $B$ and~$c$, and associated constants such as~$\r_2$, and do not vary,
either with~$N$, or with the choices made for the quantities~$\e\uii$, $1\le i\le 4$,
appearing in Lemmas \ref{d-lem-1} and~\ref{d-lem-2}. Constants involving the symbol~$\d$ are
typically to be chosen suitably small, but again only with reference to the functions $A$, $B$
and~$c$, and to associated constants such as~$\r_2$.

\subsection{Main results}\label{results}
Under these assumptions, we carry out a programme indicated in Barbour (1980),
but now in more general circumstances. We first show that the initial behaviour of $Nx_N\ut$
is well approximated by that of a supercritical $d_2$-type Markov
branching process~$Z$,
defined at the beginning of Section~\ref{phase-1},
whose mean growth rate matrix is~$B_0^T$.
Let~$\bu^T$ be the left eigenvector of~$B_0^T$ corresponding to~$\b_0$,
normalized so that $\bu^T\bone = 1$, and let the corresponding right eigenvector be~$\bv$,
normalized so that $\bu^T\bv = 1$. 
Then branching process theory (Athreya \&~Ney 1972, Chapter~V.7, Theorem~2) implies that $Z(t)e^{-\b_0 t}
\to W\bu$ a.s. as $t\to\infty$, where the random variable~$W$ has mean $Z_0^T\bv$ and satisfies $W>0$ on
the set of non-extinction, and in consequence, for as long as this approximation holds,
\eq\label{initial-stoch}
   x_N\ut(t)\ \approx\ N^{-1}e^{\b_0\{t + \b_0^{-1}\log W\}}\bu;
\en
the results that we use are proved in the appendix.
The development of~$\xi_N\ut$, the second group of components of the solution of the deterministic equation, 
also initially parallels that of~$x_N\ut$, in that the
linear approximation to~\Ref{deterministic} near~$x_0$ yields
\eq\label{initial-det}
   \xi_N\ut(t)\ \approx\ e^{B_0 t} N^{-1}Z_0 \ \sim\ N^{-1} e^{\b_0 t} (\bv^T Z_0)\,\bu
    \Eq N^{-1}e^{\b_0\{t + \b_0^{-1}\log (\bv^T Z_0)\}}\bu,
\en
by virtue of the Perron--Frobenius theorem (Seneta~2006, Theorem~2.7).
The quantity~$W$ in~\Ref{initial-stoch} is
replaced in~\Ref{initial-det} by its expectation, so that, apart from the random
time shift $\b_0^{-1}(\log W - \log \ex W)$, the two paths are much the
same.  This simple description of the development of~$x_N$ turns out to be true also 
if {\it all\/} components, and not just those of the second group, are considered;  the 
formal statement of this, together with some estimate of the accuracy of the
approximation, is the main message of Theorem~\ref{combined}.
Note that the approximations \Ref{initial-stoch} and~\Ref{initial-det}
need~$t$ to be large, so that in the first case the branching asymptotics
and in the second the Perron--Frobenius asymptotics give good approximations.
On the other hand, $t$ should not be so large as to invalidate the linearizations around~$x_0$,
implicit in both approximations.
It is the need to satisfy both requirements simultaneously, with sufficient accuracy,
and for large enough values of~$t$, that provides a major source of complication in the proofs.

In Section~\ref{phase-1}, we show that
the branching approximation in fact holds good {\it in total variation\/} up to
a time~$\t_{N,\a}^x$, chosen so that~$N\bv^T x_N\ut(\t_{N,\a}^x)$ is approximately~$N^{1-\a}$,
for any $\a > 1/3$.  As is shown by example in Section~\ref{complements} of the appendix, approximation
in total variation is typically not accurate for $\a \le 1/3$, but it is essential
to the subsequent argument that we can take $\a < 1/2$;  
we take~$\a=5/12$ for the remaining development.
If the branching process is absorbed in~$0$, then
so too, with high probability, is~$x_N\ut$.  If not, then we show that $x_N(\t_{N,\a}^x)$
is close to~$\xi(t_{N,\a}^\xi)$, where $t_{N,\a}^\xi = \b_0^{-1}(1-\a)\log N + O(1)$
is the approximate time~$t$ at which the deterministic
solution~$\xi_N$ starting in~$x_N(0)$ satisfies $\bv^T\xi_N(t) = N^{1-\a}$.
The details are to be found in Proposition~\ref{phase-1-conc}.

In Section~\ref{phase-2}, we show that the deterministic and stochastic paths $\txi_N$ and~$\tx_N$,
both starting at~$x_N(\t_{N,5/12}^x)$, and with time argument re-starting at~$0$,
stay asymptotically close for large~$N$ until an
elapsed time $t_N(\d)$, at which $\bone^T\txi_N$
first attains the value~$\d$, for a small but fixed $\d > 0$; note that
$t_N(\d)=  \b_0^{-1}\a\log N + O(1)$.
The details are given in Proposition~\ref{phase-2-conc}; the fact that $\a < 1/2$ is
needed to maintain the accuracy of approximation up to times at which the 
second components of the paths have attained asymptotically non-negligible size. From this
point onwards, Kurtz's~(1970) theorem, together with the Lipschitz continuity of the
solutions of the deterministic equations with respect to their initial conditions,
can be used to justify the further deterministic
approximation to~$x_N$,
as long as the deterministic curve remains within some fixed,
compact subset of $\reals_+^d$.
Thus~$x_N$ closely follows the deterministic path, but at a
random rate, with the randomness quickly settling down to a fixed time shift
of order~$O(1)$.  The combined theorem is as follows; the parts not justified
by theorems in Kurtz~(1970, 1977) are proved in the following sections.
For the statement of the theorem, we make the following general definitions:
\eqa
   \t^Z(0) &:=& \inf\{t>0\colon\, Z(t)=0\};\quad \t_N^{x}(0) \Def \inf\{t>0\colon\, x_N\ut(t)=0\};
                      \phantom{XXX} \label{AB-tau-0-defs}\\
   \t_{N,\a}^Z &:=& \inf\{t\colon\,\bv^T Z(t) \ge N^{1-\a} + \bv^T Z_0\}; \label{AB-tau-Z-def}\\
   \t_{N,\a}^{x} &:=& \inf\{t\colon\,\bv^T Nx_N\ut(t) \ge N^{1-\a} + \bv^T Z_0\}; \label{AB-tau-x-def} \\
   t_{N,\a}^\xi &:=&  \b_0^{-1}\{(1-\a)\log N - \log(\bv^T Z_0)\}, \label{AB-t-xi-def}
\ena
with the infimum of the empty set taken equal to infinity, and, for the particular choice
$\a = 5/12$, we define
\eq\label{AB-tau-star-defs}
   \t_{N*}^Z \Def \t_{N,5/12}^Z;\quad \t_{N*}^x \Def \t_{N,5/12}^x;\quad t_{N*}^\xi \Def t_{N,5/12}^\xi.
\en
For the Markov branching process~$Z$, defined at the beginning of Section~\ref{phase-1}, we
set $W := \lim_{t\to\infty}\bv^T Z(t)e^{-\b_0 t}$.

\begin{thm}\label{combined}
With the assumptions and definitions of Section~\ref{assumptions},
suppose that $|x_N\ui(0) - x_0\ui| \le N^{-5/12}$ and that $x_N\ut(0) = N^{-1}Z_0$ for
fixed $0 \neq Z_0 \in \Z_+^{d_2}$. Then, except on an event~$E_{N1}^c$ of asymptotically 
negligible probability, the paths of~$Nx_N\ut$ and of~$Z$ can
be coupled so as to be identical until the time~$\min\{\t^Z(0),\t_{N*}^Z\},$
in which case $\t_{N*}^Z = \t_{N*}^x = \b_0^{-1}\{\tfrac7{12}\log N - \log W\} + O(N^{-7/48})$.

Let~$\kk$ be any fixed compact subset of~$\reals_+^d$. Suppose that~$T$ is such
that $\xi_N(t_{N*}^\xi + t) \in \kk$ for all $0 \le t \le T$, where~$\xi_N$ denotes the
solution to the deterministic equation starting with $\xi_N(0) = x_N(0)$.
Then there exists a constant $\g > 0$,  a constant $k_T < \infty$ and an event~$E_{N2}^T$ such that,
on $\{\t_{N*}^x < \infty\}\cap E_{N1}\cap E_{N2}^T$,
\eq\label{deterministic-approximation-final}
   \sup_{0\le t\le \tfrac5{12}\b_0^{-1}\log N + T}|x_N(\t_{N*}^x + t) - \xi_N(t_{N*}^\xi + t)| \Le k_T  N^{-\g},
\en
and $\lim_{N\to\infty} \pr[E_{N2}^T \giv \{\t_{N1}^Z < \infty\}\cap E_{N1}] = 1$.
\end{thm}

\nin
The proof of the branching approximation is given in Section~\ref{phase-1}, and its
content summarized in Proposition~\ref{phase-1-conc}. The proof of the subsequent deterministic
approximation, up to a time at which~$x_N$ is away from the boundary, is given
in Section~\ref{phase-2}, and its content summarized in Proposition~\ref{phase-2-conc}.  The
extension to further choices of~$T$ follows from Kurtz~(1970, Theorem~(3.1)), and approximation
is then by a non-degenerate path.
There is no universal choice possible for the exponent~$\g$ appearing
in~\Ref{deterministic-approximation-final},
which is a reflection of the greater delicacy required for the approximations derived here than
in the setting of Kurtz~(1970), when any $\g < 1/2$ would satisfy;  we give an example to
illustrate this in Section~\ref{complements} of the appendix.

\medskip
Theorem~\ref{combined} can be interpreted
in the sense that, to a first approximation, the random process~$x_N$ follows the deterministic
curve starting at the same point, but with a random delay of
$\t_{N*}^x - t_{N*}^\xi \sim \b_0^{-1}\{\log(\bv^T Z_0) - \log W\}$.
The initial
condition for~$Z_0$ could be allowed to depend on~$N$, in which case the distribution of~$W$
would depend on~$N$, too:  if $|Z_0\uN| \to \infty$, then $\log(\bv^T Z_0\uN) - \log W\uN
\tod 0$, so that, to this level of approximation, the initial randomness would disappear.

\ignore{
The supremum in \Ref{deterministic-approximation-final} can of course be extended to cover
negative $t > -\min\{\t_{N*}^x,t_{N*}^\xi\}$ also, because~$x_N\ut(u)$ is uniformly close to~$0$
for $u < \t_{N*}^x$, and~$x_N\ui(u)$
is uniformly close to~$x_0\ui$, and similar statements are true of the deterministic approximations.
However, the branching approximation gives much more precise information about~$x_N\ut$
for such times.
}

\subsection{Absorption}\label{absorption}
Our motivating example actually contains two periods in which the process is close to
a boundary, the second being when the wild type becomes extinct.  The setting is then
almost exactly as in Section~\ref{assumptions},
except for the fact that the deterministic solution converges to zero
in some of its coordinates, instead of moving away from zero. 
In the notation of Section~\ref{assumptions}, this corresponds to having the largest real part 
among the eigenvalues of~$B_0$ being negative;  we denote it by~$-\b_1$.  In this setting, we also assume
that the eigenvalues of~$C$ all have negative real parts.

Under these modified assumptions,  we consider stochastic and
deterministic processes $\txi_\d$ and~$x_{N,\d}$ that are
started close to one another, as is implied by the previous results, at a point where they
are reasonably close to the stable equilibrium~$x_0$. To be more precise, we first suppose that
$|x_{N,\d}(0) - x_0| \le \d$, and that $\xi_\d$ is the solution to the deterministic equations
with $\xi_\d(0) = x_{N,\d}(0)$.  We then show that, for~$\d$ chosen small
enough, the two processes
remain close for a further time $t_N(\d) := \b_1^{-1}(\log\d + \tfrac5{12}\log N)$, at which point
the second group of coordinates, those that are converging to zero,
are of magnitude approximately~$N^{-5/12}$, 
and the first coordinates are at a similar distance from~$x_0\ui$.  We also show that, 
if $|\xi_\d(0) - \txi_\d(0)| = O(N^{-\g})$ for some $\g > 0$, then, for~$\d$ chosen small enough,
$|\xi_\d\ui(t_N(\d)) - \txi_\d\ui(t_N(\d))| = O(N^{-\g-\e})$ for some $\e>0$, and
$N^{5/12}|\xi_\d\ut(t_N(\d)) - \txi_\d\ut(t_N(\d))| = O(N^{-\g/2})$.  After this time, the
process~$(Nx_N\ut(t_N(\d) + t),\,t\ge0)$
is well approximated by a branching process~$Z$ in total variation, with rates as before.
The following theorem summarizes these results; the proofs are given in Section~\ref{phase-3}
in the appendix.

\begin{thm}\label{combined-2}
Suppose that the assumptions of Section~\ref{assumptions} hold, with the above modifications. 
Then there exist $\d > 0$ and an event~$E_N$, whose complement
has asymptotically negligible probability, such that, on~$E_N$, if $|x_{N,\d}(0) - x_0| \le \d$,
and if $|x_{N,\d}(0) - \txi_{N,\d}(0)| = O(N^{-\g_1})$ for some~$\g_1 > 0$, then
\eqs
    \sup_{0\le t\le t_N(\d)}|x_{N,\d}\ui(t) - \txi_{N,\d}\ui(t)| &\le& \tk\ui  N^{-\g};\\
   \sup_{0\le t\le t_N(\d)}\{e^{\b_1t}|x_{N,\d}\ut(t) - \txi_\d\ut(t)|\}
                 &\le& \tk\ut  N^{-\g},
\ens
with $t_N(\d) := \max\{\b_1^{-1}(\log\d + \tfrac5{12}\log N),0\}$ and
for suitable $\tk\ui,\tk\ut$ and $\g > 0$.  After~$t_N(\d)$, the process
$Nx_{N,\d}\ut(t_N(\d) + \cdot)$ can be coupled to be identical until extinction to the
(now subcritical) Markov branching process~$Z$, except on an event of asymptotically
negligible probability.  In particular, for a suitable constant~$h^*$,
the time~$t_N(\d) + T_N$ at which~$x_N\ut$ is absorbed in~$0$ is such that
$ \law(\b_1 T_N - \log N - \log(\bv^T\txi_{N,\d}(t_N(\d))) - \log(h^*))$
converges in total variation as $N\to\infty$ to a Gumbel distribution.
\end{thm}

\medskip
The approximation given by Theorem~\ref{combined-2} shows
that, to a first approximation, the random process~$x_N$ follows the deterministic
curve starting at the same point until the time $t_N(\d) = \b_1^{-1}(\log\d + \tfrac5{12}\log N)$.
The law of large numbers for~$Z$ starting at $Nx_{N,\d}\ut(t_N(\d))$ then shows that
the same is true afterwards, though since, for such times, $x_{N,\d}\ut$ is uniformly small
and $x_N\ui(t)$ is close to~$x_0$, the conclusion is of little interest.  However, the
branching approximation delivers more detailed information.  In particular, the time
taken by the deterministic solution~$\txi_{N,\d}$ from $t_N(\d)$ until
$\hat t_N := \inf\{t > 0\colon\, \bv^T\txi_{N,\d}(t) = N^{-1}\}$ is such that
\[
    \hat t_N \ \sim\  \b_1^{-1}\{\log N + \log(\bv^T\txi_{N,\d}(t_N(\d)))\}.
\]
The deterministic solution itself never reaches $\txi_{N,\d}\ut = 0$ in finite time,
but~$\hat t_N$ is the sort of approximation that might be made for the time to extinction,
based on deterministic considerations.  Theorem~\ref{combined-2} shows that this is
reasonable, but that the duration in the stochastic model has an additional random component
$\b_1^{-1}\{G + \log(h^*)\}$.

\subsection{The bare bones example}\label{example}
These results can all be applied to the bare bones example discussed earlier, which
is of the form proposed in Section~\ref{assumptions}, with $d_1 = d_2 = 1$.  
In the initial stages, the matrices $A(x)$ and~$B(x)$
are the scalars $-\g x_1$ and \hbox{$(a_2 - \g x_1 - x_2)$,} and the function
$c(x_1) = -a_1(x_1 - a_1) - (x_1-a_1)^2$, so that we have \hbox{$C = -a_1 < 0$} and
\hbox{$\tc(x_1) = - (x_1-a_1)^2$.}  Assuming that $a_2 > \g a_1$, the unstable equilibrium of the
deterministic equations is $x_0 = (a_1,0)^T$, and $\b_0 = B(x_0) = a_2 - \g a_1 > 0$.
The set~$\jj_2$ consists of the transitions $\{(0,1),(0,-1)\}$, and $s(J) = 2$ for both
of them; the corresponding functions~$\bg^J$ are $a_2$ and $(\g x_1 + x_2)$ respectively.
The process~$Z$ is a linear birth and death process,
with {\it per capita\/} birth and death rates $a_2$ and~$\g a_1$ respectively, and
its behaviour is well understood.  In particular, the limiting random variable~$W$,
conditional on the event of non-extinction, has a Gamma distribution ${\rm Ga\,}(Z_0,1)$.
Hence, if~$Z_0=1$, the delay in following the deterministic curve, given in general by
\[
     \t_{N*}^x - t_{N*}^\xi\ \sim\ \b_0^{-1}\{\log(\bv^T Z_0) - \log W\},
\]
has the distribution of $\{a_2 - \g a_1\}^{-1}G_1$, where~$G_1$ has a
Gumbel distribution.

For the latter stages of the example, in the case when $a_1 < \g a_2$,
the wild type individuals eventually die out.  To match the formulation
in Section~\ref{absorption},
it is necessary to swap the meaning of the coordinates,
so that the second coordinate now represents the remaining numbers of wild type individuals.
The matrices $A(x)$ and~$B(x)$
become the scalars $-\g x_2$ and \hbox{$(a_1 - \g x_2 - x_1)$,} and the function
$c(x_1)$ is given by $-a_2(x_1 - a_2) - (x_1-a_2)^2$, so that we obtain $C = -a_2 = -\k$ and
\hbox{$\tc(x_1) = - (x_1-a_2)^2$.}    The strongly stable equilibrium of the
deterministic equations with the mutants established is given by
$x_0 = (a_2,0)^T$, and $-\b_1 = B(x_0) = a_1 - \g a_2 < 0$.
The set~$\jj_2$ consists as before of the transitions $\{(0,1),(0,-1)\}$, and $s(J) = 2$ for both
of them; the corresponding
functions~$\bg^J$ are $a_1$ and $(\g x_1 + x_2)$ respectively.
The branching process~$Z$ is again a linear birth and death process,
with {\it per capita\/} birth and death rates $a_1$ and~$\g a_2$ respectively.
For this process, the constant~$h^*$ appearing in the final approximation
can be evaluated, using the definition in Heinzmann~(2009, p.299), as
$1 - a_1/(\g a_2)$.  Combining this with the above, we can deduce that the
asymptotics of the entire time from the introduction of a single mutant until the
extinction of the wild type individuals is given by
\[
     \{a_2 - \g a_1\}^{-1}G_1 + \{\g a_2 - a_1\}^{-1}\{\log(1 - a_1/(\g a_2)) + G_2\}
     + T\uN,
\]
where~$T\uN = (\{a_2 - \g a_1\}^{-1} + \{\g a_2 - a_1\}^{-1})\log N + O(1)$ is the time taken
for the deterministic curve to get from the initial state, where the proportion of
mutants is~$N^{-1}$, to the state in which the proportion of wild type individuals is~$N^{-1}$;
and $G_1$ and~$G_2$ are independent Gumbel random variables.  The duration of the
closed stochastic epidemic, studied in Barbour~(1975), could also be approached in
a similar way.  In that example, however, the function~$c$ is identically
zero, so that the final stages have to be treated differently.

\section{The deterministic solutions}\label{deterministic-section}
\setcounter{equation}{0}
For use in our arguments, we collect some properties of the solutions to the
deterministic equations in the neighbourhood of the initial point, deferring the
proofs of the lemmas to the appendix.  We first
use variation of constants to rewrite the equations in the form
\eqa
    \xi\ui(t) &=& \xi\ui(0) + \int_0^t \{A(\xi(u)) \xi\ut(u) + c(\xi\ui(u))\}\,du ; \label{3-d}\\
        &=& x_0\ui + e^{Ct}(\xi\ui(0) - x_0\ui) \non\\
      &&\quad\mbox{} +
       \int_0^t e^{C(t-u)}\{A(\xi(u)) \xi\ut(u) + \tc(\xi\ui(u))\}\,du ;
                  \phantom{XX} \label{3a-d}\\
         \xi\ut(t) &=& \xi\ut(0) + \int_0^t B(\xi(u)) \xi\ut(u)\,du \non\\
       &=& e^{B_0t}\xi\ut(0) + \int_0^t e^{B_0(t-u)} \{B(\xi(u)) - B_0\} \xi\ut(u)\,du. \label{4-d}
\ena
We recall that, in the arguments that follow, constants involving the symbol~$k$
do not vary with the choices made
for the quantities~$\e\uii$, $1\le i\le 4$.  In our applications, these
quantities become small, as~$N$ increases, as negative powers of~$N$, and the assumptions
made about them in the lemmas are automatically satisfied for all~$N$ sufficiently
large.  For use in what follows, define
\eq\label{t-delta-eps-def}
    t_0(\d,\e) \Def \b_0^{-1}\log(\d/\e);\qquad t_1(\d,\e) \Def \b_1^{-1}\log(\d/\e),
\en
for $\d \ge \e > 0$, where~$\b_0$ is as in Section~\ref{assumptions} and~$\b_1$ is
as in Section~\ref{absorption}.

\begin{lem}\label{d-lem-1}
Under the assumptions of Section~\ref{assumptions}, there exists a~$\d_0$ with $0 < \d_0 \le 1$,
depending only on the functions $A,B$ and~$c$ and associated constants such as~$\r_2$, with
the following properties.
If~$\xi$ satisfies the equations \Ref{3a-d} and~\Ref{4-d}, with initial condition such
that $|\xi\ui(0) - x_0\ui| \le \e\ui$ and $|\xi\ut(0)| = \e\ut$, and if 
\eq\label{e-i-bnd}
     4\cst_1\e\ui \Le  \min\{1,(\r_2/4)\}\quad\mbox{and}\quad \e\ut \Le \d_0,
\en
and if also 
\eq\label{e-e-bnd}
     \e\ui\log(1/\e\ut) \Le \min\{1, \b_0/(24\cst_2\cst_1 \|DB\|_{\r_2}), \b_0/(32\kkkc \cst_1^2)\} ,
\en
then, for all~$0 \le t \le t_0(\d_0,\e\ut)$,
\eqa
  \sup_{0\le u\le t}|\xi\ui(u) - x_0\ui| &\le& k\ui\{\e\ui + \e\ut e^{\b_0 t}\}; \non \\
   \sup_{0\le u\le t} e^{-\b_0 u}|\xi\ut(u)| &\le& k\ut\e\ut ;\non \\
   \sup_{0\le u\le t} e^{-\b_0 u}|\xi\ut(u) - e^{B_0u}\xi\ut(0)|
         &\le& k\uh\e\ut\{\e\ui\log(1/\e\ut) + \e\ut e^{\b_0t}\},
       \non
\ena
for suitable $k\ui,k\ut,k\uh$.  Furthermore, if~$\txi$ satisfies \Ref{3a-d} and~\Ref{4-d} with
initial condition~$\txi(0)$ satisfying $|\txi\ut(0) - \xi\ut(0)| \le \e\uh \le k\uf(\e\ut)^{1+\g}$
and $|\txi\ui(0) - \xi\ui(0)| \le  k\uv\e\ut\log(1/\e\ut)$,
for some $k\uf,k\uv > 0$ and $0<\g<1$, then there exist $k^{(6)},k^{(7)}, k^{(8)}$ and~$0 < \d_1 \le \d_0$
such that, for all~$\e\ut \le \min\{k^{(8)},\d_1\}$,
\eqa
    \sup_{0\le u\le t_0(\d_1,\e\ut)}|\xi\ui(u) - \txi\ui(u)|  &\le& k^{(6)}(\e\ut)^{\g/2};
                    \non\\
    \sup_{0\le u\le t_0(\d_1,\e\ut)}\{e^{-\b_0 u}|\xi\ut(u) - \txi\ut(u)|\}  &\le& k^{(7)} (\e\ut)^{1+\g/2}.
                    \non
\ena
Here, $\d_1$ may depend on the choice of~$\g$, as well as on the functions $A,B$ and~$c$.
\end{lem}

\medskip
We also consider the final stages of such a process, before
absorption in a strongly stable equilibrium with the $2$-components equal to zero.  Under
such circumstances, we can still work under assumptions
similar to those made in Section~\ref{assumptions}.
The main difference is to require that the eigenvalue of~$B_0$ with largest real part is negative;
we denote it by~$-\b_1$.  We also assume that the equilibrium~$x_0$ is strongly
attractive, in the sense that
\eq\label{matrix-bnd-attract}
    |e^{Ct}x| \Le \cst_1 e^{-\k t}|x|,\quad x \in \reals^{d_1},\ t\ge0,
\en
for some $\k > 0$ and $\cst_1 < \infty$;  the previous assumptions of Section~\ref{assumptions}
only required $\k \ge 0$.
The analogue of Lemma~\ref{d-lem-1} is then as follows.

\begin{lem}\label{d-lem-2}
With the assumptions of Section~\ref{assumptions}, modified as in Section~\ref{absorption},
let $\xi_\d$  satisfy the equations \Ref{3a-d} and~\Ref{4-d}, with $\xi_\d(0) =: x_{\d0}$
such that $|x_{\d0} - x_0| \le \d$.  Then, for any $0 < \k' < \min\{\k,\b_1\}$,
there exists a $\d_0 > 0$ and constants $\hk\uii$ such that, for all~$0 < \d \le \d_0$,
\eqs
  \sup_{u \ge 0} e^{\k'u}|\xi_\d\ui(u) - x_0\ui| &\le& \hk\ui\d;\\
  \sup_{u \ge 0} e^{\b_1 u}|\xi_\d\ut(u)| &\le& \hk\ut\d ;\\
  \sup_{u \ge 0} e^{\b_1 u}|\xi_\d\ut(u) - e^{B_0u}x_\d\ut(0)| &\le& \hk\uh\d^2.
\ens
Furthermore, for any $\th > 0$, there exists a $\d(\th) > 0$ such that,
for any $0 < \d \le \d(\th)$, if~$\txi_\d$ satisfies \Ref{3a-d} and~\Ref{4-d}
with~$\txi_\d(0)$ satisfying $|\txi_\d(0) - x_{\d0}| \le \e\uf$, and if
$0 < \h < \d$ and $\e\uf\h^{-\th} \le \KKK$, for~$\KKK$ defined implicitly in~\Ref{eps-restn},
then
\eqa
    \sup_{0\le u\le t_1(\d,\h)} |\xi_\d\ui(u) - \txi_\d\ui(u)| &\le& \hk^{(5)}\e\uf\h^{-\th};
                             \label{delta-3-d}\\
    \sup_{0\le u\le t_1(\d,\h)}\{e^{\b_1 u}|\xi_\d\ut(u) - \txi_\d\ut(u)|\}  &\le& \hk^{(6)} \e\uf\h^{-\th},
                    \label{delta-4-d}
\ena
for suitable $\hk^{(5)}$ and $\hk^{(6)}$.
\end{lem}

Note that 
the estimates made in the discussion preceding Theorem~\ref{combined-2} can be justified by
the final statements of Lemma~\ref{d-lem-2}.  Taking $\e\uf = O(N^{-\g})$ for some $\g > 0$ and
$\h = N^{-5/12}$, choose~$\th$ such that $\th < \max\{\k'/\b_1, 6\g/5\}$.

\section{The branching approximation}\label{phase-1}
\setcounter{equation}{0}
In this section, we establish the approximation to $N x_N\ut$ by a Markov branching process~$Z$
in the early stages,
starting with $x_N(0) = x_{N,0}$ such that $x_{N,0}\ut = N^{-1}Z_0$ and
$|x_{N,0}\ui - x_0\ui| \le \e_N\ui := N^{-\a}$, for $\a > 1/3$.
The process~$Z$ is obtained by replacing $\bg^J(x_N(t))$ by~$\bg^J(x_0)$ in the transition rates
which have $J\ut \neq 0$, and by taking its corresponding jumps to be~$J\ut$.  It
is a Markov branching process; for each~$J$ such that $J\ut \neq 0$, an individual of
type~$s(J)$ gives birth to~$J_i$ individuals of type~$i$, $d_1 < i \le d$,
(if $J_{s(J)} = -1$, this represents the death of an individual
of type~$s(J)$) with {\it per capita\/} rate~$\bg^J(x_0)$.  It is thus
natural to index the components of~$Z$ by $\{d_1+1,\ldots,d\}$,
to match the indexing in~$X_N$; we denote the resulting state space of~$Z$ by $\zz$.
For $z \in \zz$, let
\eq\label{AB-q-def}
   q^J(z) \Def  \bg^J(x_0)z_{s(J)}\quad\mbox{and}\quad q(z) \Def \sjjt q^J(z);
\en
then, if~$Z$ is in state~$z$, the time until its next jump is distributed as~$\Exp(q(z))$, and
the probability that it is a $J$--transition, causing a corresponding change of~$J\ut$ in~$Z$,
is $q^J(z)/q(z)$, $J\in\jj_2$.  Since there are only finitely many $J \in \jj$, the means
and covariances of the offspring distributions of individuals of the different types are
all finite.  In particular, as noted in Section~\ref{results}, the mean growth rate matrix
is given by~$B_0^T$, whose positive left and right eigenvectors $\bu^T$ and~$\bv$ are normalized
so that $\bu^T \bone = \bu^T \bv = 1$.
Our approximation shows that, except on an event
of negligible probability, the process~$Nx_N\ut$ can be constructed
so as to have paths identical to those of~$Z$, up to the time~$\t_{N1}^Z$ at which, if ever, $\bv^T Z$
has grown by at least the amount $N^{1-\a}$ from
its initial value of $\bv^T Z_0$.  The full details are given below in Proposition~\ref{phase-1-conc}.

We begin by considering the first components~$x_N\ui(\cdot)$ of~$x_N$.
Under our assumptions on~$F$, they satisfy
the equation
\[
     dx_N\ui(t) \Eq A(x_N(t))x_N\ut(t) + C(x_N\ui(t) - x_0\ui) + \tc(x_N\ui(t)) + dm_N\ui(t),
\]
where~$m_N$ is as defined in~\Ref{AB-mN-def},
and this can be integrated by variation of constants to give
\eqa
    x_N\ui(t) &=& x_0\ui + e^{Ct}(x_{N,0}\ui - x_0\ui) \non\\[1ex]
    &&\quad\mbox{} + \int_0^t e^{C(t-u)}\{A(x_N(u)) x_N\ut(u) + \tc(x_N\ui(u))\}\,du \non\\
      &&\qquad\mbox{}\ + m_{N}\ui(t) + C\int_0^t e^{C(t-u)} m_N\ui(u)\,du;  \label{1-new}
\ena
note that
\[
    m_{N}\ui(t) + C\int_0^t e^{C(t-u)} m_N\ui(u)\,du \Eq \int_0^t e^{C(t-u)} dm_N\ui(u),
\]
explaining the stochastic term in~\Ref{1-new}.  For~$x_N\ut$, up to the time
at which it has made~$n(N)$ jumps, it is enough for now to know
that it is bounded by $N^{-1}\{|Z_0| + J^*n(N)\}$, where $J^* := \max_{J\in\jj_2}|J|$.

We first use~\Ref{1-new} to show
that~$x_N\ui(t)$ moves away from~$x_0\ui$ rather
slowly.  For this, it is necessary to show that~$|m_N|$ remains uniformly small with
high probability for a long enough time interval.  This is the substance
of the following lemma. To state it, we define
\eq\label{tau-def}
    \t_N \Def \inf\{t > 0\colon\, |x_N(t)- x_0| > \r_2\},
\en
and use~$\pr^0$ to denote probabilities given $x_N(0) = x_{N,0}$.

\begin{lem}\label{MG-control}  
Let $T_N := k\log N$ for some $k> 0$, and define
\[
    E_N(k) \Def \Bigl\{ \sup_{0\le t\le T_N\wedge\t_N}|m_N(t)| \le  \h_{N1}(k) \Bigr\},
\]
where $\h_{N1}(k) := \Bigl(2\sqrt{k}\sjj |J|\Bigr) N^{-1/2}(\log N)^{3/2}$.
Then $\pr^0[E_N(k)^c] = O(N^{-r})$ for any $r > 0$.
\end{lem}

\begin{proof}
Let~$E_N'(k)$ denote the event
\eq\label{good-event-k}
    E_N'(k) \Def \Bigl\{\max_{J\in\jj}\sup_{0\le t\le T_N\wedge\t_N}|N^{-1}P^J(NG_N^J(t)) - G_N^J(t)|
                                      \le 2N^{-1/2}\sqrt{T_N} \log N \Bigr\}.
\en
Note that the quantities~$t^{-1}G_N^J(t)$ are uniformly bounded in~$t \le \t_N$, because the
functions~$g^J$ are continuous and~$x_N(t)$ is restricted to a compact set for such~$t$.
Denoting this bound by~$g^*$, it follows from the Chernoff inequalities that, for~$N$ such
that $T_N \ge 1$,
\eq\label{mN-bnd-1-k}
   \pr[E_N'(k)^c]
      \Le 2|\jj|N^{1/2} g^* T_N N \exp\{-(\log N)^2/\{2(g^*+1)\}\} \Eq O(N^{-r}),
\en
for any $r > 0$.  However, on the event~$E_N'(k)$,
\eq\label{mN-bnd-2-k}
   \sup_{0\le t\le T_N\wedge\t_N}|m_N(t)| \le \Bigl(2\sqrt{k}\sjj |J|\Bigr) N^{-1/2}(\log N)^{3/2},
\en
so that $E_N(k) \supset E_N'(k)$, which, with~\Ref{mN-bnd-1-k}, proves the lemma.
\end{proof}

\medskip
Now define $\t_1(m) := \inf\{t > 0\colon\, |x_N\ut(t)| > m/N\}$, and write
\[
    d_N\ui(t,m) \Def \sup_{0\le u\le t\wedge\t_1(m)}|x_N\ui(u) - x_0\ui|.
\]

\begin{lem}\label{close-start}
With the assumptions and notation of Section~\ref{assumptions},
fix any $k > 0$, and assume that~$N$ is large enough so that
$$
   k\log N \max\{\cst_1|x_{N,0}\ui - x_0\ui|,\h_{N1}'(k)\} \le 1/(40\cst_1 \kkkc ),
$$
where $\h_{N1}'(k) := \h_{N1}(k)(1 + \cst_1\|C\|k\log N)$.  Suppose that~$E_N(k)$
occurs.
Then, for all $0 \le t \le k\log N$ and $m < N/\{20k^2\cst_1^2\kkkc \|A\|_{\r_2}(\log N)^2\}$,
\[
    d_N\ui(t,m) \Le \tfrac87\{\cst_1(|x_{N,0}\ui - x_0\ui| + t\|A\|_{\r_2}(m/N)) + \h_{N1}'(k)\}.
\]
\end{lem}

\begin{proof}
From equation~\Ref{1-new} and the assumptions on~$C$ and~$\r_2$, and from the definition of~$E_N(k)$,
it follows immediately that, for $t \le (\t_1(m)\wedge k\log N) $ such that also
\eq\label{d-1-integral-cond1}
      \cst_1 \kkkc  \int_0^t d_N\ui(u,m)\,du \Le \frac18,
\en
we have
\eqa
   |x_N\ui(t) - x_0\ui| &\le& \!\!\cst_1 |x_{N,0}\ui - x_0\ui| + \h_{N1}(k)\phantom{XXXXXXXXXXXXXXXX}\non \\
        &&\mbox{}\!\!   + \cst_1\int_0^t
    \Bigl\{\|A\|_{\r_2}\frac{m}N + \kkkc  \{d_N\ui(u,m)\}^2 +\|C\|\h_{N1}(k) \Bigr\}\,du \non \\
      &\le& \!\! \cst_1\Bigl\{|x_{N,0}\ui - x_0\ui| + t\|A\|_{\r_2}\frac{m}N\Bigr\}
            +  \tfrac18 d_N\ui(t,m)  + \h_{N1}'(k).
     \label{lemma-result}
\ena
Now
for~$t \le (\t_1(m)\wedge k\log N)$, \Ref{lemma-result} implies that
\eqs
  \lefteqn{\cst_1 \kkkc  \int_0^t d_N\ui(u,m)\,du} \\
    &&\Le \cst_1 \kkkc  \,\tfrac87
     \{\cst_1(t|x_{N,0}\ui - x_0\ui| + \tfrac12 t^2 \|A\|_{\r_2}(m/N)) + t\h_{N1}'(k)\} \\
    &&\Le \tfrac87\{\tfrac1{40} + \tfrac 1{40} + \tfrac1{40}\} \Eq \tfrac3{35}\ <\ \tfrac18 ,
\ens
the bound assumed in~\Ref{d-1-integral-cond1}.
Hence, since $\int_0^t d_N\ui(u,m)\,du$ is continuous in~$t$, we can apply
Lemma~\ref{AB-little-lemma} with $\DDD = 0$ to show that, for all~$N$ sufficiently large,
the inequality~\Ref{lemma-result} holds for all~$t \le (\t_1(m)\wedge k\log N)$, and the
lemma is proved.
\end{proof}

\medskip
Since, in the early phase, $x_N\ut(t) \approx 0$ and Lemma~\ref{close-start} shows
that $x_N\ui(t) \approx x_0\ui$, the process~$Nx_N\ut$ can plausibly be well approximated by replacing
$\bg^J(x_N(t))$ by~$\bg^J(x_0)$ in its transition rates, obtaining the Markov
branching process~$Z$.  To show that this is indeed the case, we consider a path starting in~$Z_0$,
having $J_1,\ldots,J_n$ as its first~$n$ transitions and $t_1,\ldots,t_n$ their times. Then the probability
density of this path segment is given by
\eq\label{density-0}
   \prod_{l=0}^{n-1} \Bl \exp\{-(t_{l+1}-t_l)q(z_l)\}  q^{J_{l+1}}(z_l) \Br,
\en
where $z_l := Z_0 + \sum_{i=1}^l J_i\ut$ and $t_0=0$, and the functions $q$ and~$q^J$ are as in~\Ref{AB-q-def}.

The corresponding expression for~$Nx_N\ut$ is
more complicated, since the process is only Markovian if the state space is extended to
include all the original coordinates.  Defining
\eqs
   q_N^J(x\ui,z) &:=& \bg^J([x\ui,N^{-1}z]) z_{s(J)};\qquad
   q_N(x\ui,z) \Def \sjjt q_N^J(x\ui,z),
\ens
for $x \in \reals_+^d$, $z \in \zz$ and $J \in \jj_2$, with $[y_1,y_2]$ denoting $(y_1^T,y_2^T)^T$,
and then writing
\eqs
   H^J(x\ui,z,t)
   &:=& \ex^{(x\ui,z)}\Bl \exp \Blb -\int_0^t q_N(x_N\ui(u), z)\,du \Brb
                  q_N^J(x_N\ui(t),z)  \Br,
\ens
the probability density at $\{(J_1,t_1),\ldots,(J_n,t_n)\}$ is given by
\eq\label{density-1}
      \ex^0\Bl \prod_{l=0}^{n-1} H^{J_{l+1}}(x_N\ui(t_l),z_l,t_{l+1}-t_l) \Br;
\en
here, $\ex^{(x\ui,z)}$ denotes expectation conditional on
$x_N(0) = \Bl\begin{array}{c} x\ui \\ N^{-1}z \end{array}\Br$,
and $\ex^0$ as before denotes expectation conditional on $x_N(0) = x_{N,0}$.
Hence the likelihood ratio, with respect to the branching process measure, of a path successively
entering the states $z_{\{1,n\}} := z_1,z_2,\ldots,z_n$ at times $t_{\{1,n\}} := t_1,\ldots,t_n$ is given by
\eq\label{LR-1}
  R_n(z_{\{1,n\}},t_{\{1,n\}}) \Def
    \ex^0\Bl \prod_{l=0}^{n-1} \tG^{J_{l+1}}(x_N\ui(t_l),z_l,t_{l+1}-t_l) \Br,
\en
where
\eqa
    \lefteqn{\tG^J(x\ui,z,t) \Def H^J(x\ui,z,t) e^{tq(z)} / q^{J}(z)} \non\\
     &=& \ex^{(x\ui,z)}\Bl \exp \Blb -\int_0^t \{q_N(x_N\ui(u), z) - q(z)\}\,du \Brb
                  \frac{q_N^J(x_N\ui(t),z)}{q^{J}(z)}  \Br.\phantom{XX} \label{tilde-G-def}
\ena

Let $\t_2(n)$ denote the time at which~$x_N\ut$ makes its $n$-th jump. Then
$|x_N\ui(u) - x_0\ui|$ remains uniformly of order
$$
    O\{N^{-1/2}(\log N)^{5/2} + N^{-\a} + n(N)\log N/N\}
$$
for $0 \le u \le (\t_2(n(N))\wedge k\log N)$, except an event of probability  $O(N^{-r})$,
for any $r>0$, because of Lemma~\ref{close-start}.  This implies that the rates
$q_N^J(x_N\ui(u),z)$ are close to~$q^J(z)$, $J\in\jj_2$, throughout the same $u$-interval.
We now use this to show
that the joint distributions of the times and values of the
first~$n(N)$ jumps of the processes $Z$ and~$Nx_N$ are close to one another.

\begin{lem}\label{phase-1-lem}
Let $x_N(0)$
be such that $|x_N\ui(0) - x_0\ui| = O(N^{-\a})$
and $x_{N}\ut(0) = N^{-1}Z_0$.  Then, for any
fixed $k > 0$ and $1/3 < \a < 1$,
the total variation distance $\dtv\uN$ between the distributions of the paths of~$Nx_N\ut$
and those of~$Z$, restricted to the first~$k N^{1-\a}$ jumps, is such that
$\lim_{N\to\infty}\dtv\uN = 0$.
\end{lem}

\begin{proof}
Letting $\t_{\{1,n\}}$ denote the times of the first~$n$ jumps of~$Z$,
the main aim is to show that the likelihood ratio $R_n(Z_{\{1,n\}},\t_{\{1,n\}})$
defined in~\Ref{LR-1} is close to~$1$ with high probability.
First, defining
\eqs  
\lefteqn{\hG^J(x\ui,z,t,y)}\\
   &:=& \ex^{(x\ui,z)}\Bl \exp \Blb -\int_0^t \{q_N(x_N\ui(u), z) - q(z)\}\,du \Brb
                  \frac{q_N^J(y,z)}{q^{J}(z)} \bone\{x_N\ui(t) = y\}  \Br,
\ens
for $y \in N^{-1}\Z_+^{d_1}$, we can express the ratio
\[
    V_{n+1}(z_{\{1,n+1\}},t_{\{1,n+1\}})
         \Def \frac{R_{n+1}(z_{\{1,n+1\}},t_{\{1,n+1\}})}{R_n(z_{\{1,n\}},t_{\{1,n\}})}
\]
as
\eq\label{single-ratio}
   V_{n+1}(z_{\{1,n+1\}},t_{\{1,n+1\}}) \Eq \ex_n\{ \tG^{J_{n+1}}(Y,z_{n+1},t_{n+1}-t_n) \},
\en
where~$\ex_n$ denotes expectation with respect to the measure
with probabilities~$p_n(y)$ given by
\[
  \frac{\ex^0\Bl \hG(x_N\ui(t_{n-1}),z_{n-1},t_n-t_{n-1},y)
            \prod_{l=0}^{n-2} \tG^{J_{l+1}}(x_N\ui(t_l),z_l,t_{l+1}-t_l) \Br}
     {R_n(z_{\{1,n\}},t_{\{1,n\}})},
\]
for $y \in N^{-1}\Z_+^{d_1}$.  Now, defining the $\s$-fields $\Sigma_n := \s(Z_{\{1,n\}},\t_{\{1,n\}})$,
the process $(R_n(Z_{\{1,n\}},\t_{\{1,n\}}),\Sigma_n,\,n\ge0)$, being a likelihood ratio, is a martingale with
expectation~$1$.  We wish to show that it stays close to its expectation with high probability.

First, we consider the process~$x_N$ obtained by replacing $\bg(x)$ with~$\bg(x_0)$ in the transition rates
for jumps $J \in \jj_2$, whenever $|x-x_0| > \th_N$, yielding a new process~$x_{N,\th}$;
the quantity $\th_N \le \r_2$ is yet to be determined.  We then conduct the whole analysis for~$x_{N,\th}$.
Observe that, in~\Ref{single-ratio}, the quantity $\tG^{J_{n+1}}(Y,z_{n+1},t_{n+1}-t_n)$ is, from
its definition~\Ref{tilde-G-def}, itself an expectation, and that, for the process~$x_{N,\th}$,
the quantity within the expectation is itself close to~$1$.  To see this, let
$Q^* := \max_{J\in\jj_2}\{\|D\bg^J\|_{\r_2}/b_*^J\}$;  then
\eqs
    \lefteqn{\Blm\exp \Blb -\int_0^t \{q_N(x_{N,\th}(u), z) - q(z)\}\,du \Brb
                  \frac{q_N^J(x_{N,\th}(t),z)}{q^{J}(z)} - 1\Brm} \\[1ex]
    &&\Le  \exp\{q(z)tQ^*\th_N\}(1 + Q^*\th_N) - 1 
     \Le Q^*\th_N\{(5/4)q(z)t e^{q(z)t/4} + 1\},
\ens
if also $\th_N \le 1/\{4Q^*\}$.  Hence, for~$x_{N,\th}$,
\eqa
   \lefteqn{\ex\{(V_{n+1}(Z_{\{1,n+1\}},\t_{\{1,n+1\}}) - 1)^2 \giv \Sigma_n\}}\non\\
   &\le& \{Q^*\th_N\}^2 \int_0^\infty q(z_n) e^{-q(z_n)t} \{(5/4)q(z_n)t e^{q(z_n)t/4} + 1\}^2 \,dt \non\\
   &\le& 27 \{Q^*\th_N\}^2. \label{variance-inequality}
\ena
Writing $\ps_N := 27 \{Q^*\th_N\}^2$ and $R_r := R_r(Z_{\{1,r\}},\t_{\{1,r\}})$,
this gives
\eqs
    \ex^0\{(R_n-1)^2\} &=& \ex^0\{(R_n-R_{n-1})^2\} + \ex^0\{(R_{n-1}-1)^2\} \\
              &\le& \ps_N\ex^0 R_{n-1}^2 + \ex^0\{(R_{n-1}-1)^2\} \\
    &=& \ex^0\{(R_{n-1}-1)^2\}(1+\ps_N) + \ps_N \Le (1+\ps_N)^n - 1.
\ens
In consequence, for the process~$x_{N,\th}$, if $n\ps_N \le 1$,
\eq\label{variance-bnd}
  \ex^0\{(R_n(Z_{\{1,n\}},\t_{\{1,n\}}) - 1)^2\} \Le ne\ps_N = 27ne \{Q^*\th_N\}^2.
\en
Now the total variation distance~$\dtv$ between probability measures $P_1$ and~$P_2$ on
a measurable space $(S,\ff)$ can be expressed as
\eqa
    \dtv(P_1,P_2) &:=& \sup_{A\in\ff}|P_1(A) - P_2(A)|
           \Eq \frac12\int_S \Blm R_{12}  - 1 \Brm \, dP_2 \non\\
      &=& -\int_S \min\Blb 0, R_{12} - 1 \Brb \, dP_2 ,\label{AB-dtv-def}
\ena
where $R_{12} := \frac{dP_1}{dP_2}$.
In view of~\Ref{variance-bnd}, it thus follows that, for~$x_{N,\th}$, $\dtv\uNt = O(n(N)^{1/2}\th_N)$.
By Thorisson~(2000, Chapter~3, Theorem~7.3 and (8.19)), this also implies that
the process~$Nx_{N,\th}$ and the branching process~$Z$ can be realized on the
same probability space in such a way that their paths coincide up to the first~$n(N)$ jumps,
except on an event of probability of order $O(n(N)^{1/2}\th_N)$.

Now fix $k > 0$, to be specified later, and,
for $m(N) :=  |Z_0| + J^*n(N)$, define
\[
    \th_N\ui \Def \tfrac87\Bigl\{\cst_1\{|x_{N}\ui(0) - x_0\ui| + k\log N\|A\|_{\r_2}(m(N)/N)\} + \h_{N1}'(k)\Bigr\}
\]
and $\th_N\ut := N^{-1}m(N)$; set $\th_N := \th_N\ui + \th_N\ut$.
Note that,
with $n(N) = O(N^{1-\a})$ for $1/3 < \a < 1$, this choice of~$\th_N$
satisfies $\th_N \le 1/\{4Q^*\}$ for all~$N$ large enough, and that the total variation distance
$\dtv\uNt$ is of small order $O\{n(N)^{1/2}(N^{-1}n(N)\log N + N^{-\a})\}$.
Now $x_N$ and~$x_{N,\th}$ can be coupled by running their paths identically until
$\t_N(\th_N) := \inf\{t>0\colon\,|x_{N,\th}(t)- x_0| > \th_N\}$.
So, for this choice of~$\th_N$, let
\[
    \s_N\ui \Def \inf\{t>0\colon\,|x_{N,\th}\ui(t)- x_0\ui| > \th_N\ui\};\quad
    \s_N\ut \Def \inf\{t>0\colon\,|x_{N,\th}\ut(t)| > \th_N\ut\}.
\]
If, for $n(N) := k_0N^{1-\a}$, we can show that
$\pr[\s_N\ui \wedge \s_N\ut < \t_{n(N)}\wedge \t_N^x(0)]$ is asymptotically
small as $N\to\infty$, where
$\t_n$ denotes the time of the $n$-th jump of $x_{N,\th}\ut$ and~$\t_N^x(0)$, as
in~\Ref{AB-tau-0-defs}, its time of first hitting~$0$, the lemma will be proved.

It is immediate from the definition of~$\th_N\ut$ that $\s_N\ut \ge \t_{n(N)}$ a.s.  Then,
by Lemma~\ref{close-start},
$ \pr[\{\s_N\ui  < \t_{n(N)}\wedge \t_N^x(0)\}\cap \{\s_N\ui \le k\log N\}] = O(N^{-r})$,
for any $k,r > 0$.  Finally,
\eqs
   \lefteqn{\pr[\{\s_N\ui  < \t_{n(N)}\wedge \t_N^x(0)\}\cap \{\s_N\ui > k\log N\}]}\\
     &&\Le \pr[\{\t_{n(N)} > k\log N\} \cap \{|x_{N,\th}\ut(k\log N)| > 0\}] \\
     &&\Le \dtv\uNt + \pr[0 < W^\bv(k\log N) \le m(N) \exp\{-\b_0k\log N\}],
\ens
where $W^\bv(t) := \bv^T Z(t)e^{-\b_0 t}$.  However, choosing any $k > \b_0^{-1}(1-\a)$, the
latter probability is asymptotically small, because $m(N) = O(N^{1-\a})$ and,
writing $\t^Z(0) := \inf\{t>0\colon\, Z(t) = 0\}$, as in~\Ref{AB-tau-0-defs},
\[
    \{\lim_{t\to\infty}W^\bv(t) = 0\} \Eq \{\t^Z(0) < \infty\}\ \mbox{a.s.},
\]
(Athreya \&~Ney 1972, Chapter~V.7, (27) in Theorem~2).  This proves the lemma.
\end{proof}

\medskip
As a result of Lemma~\ref{phase-1-lem}, for any fixed~$k$, probabilities for the paths of~$Nx_N\ut$
up to the first~$k N^{1-\a}$ jumps can, with only small error, be computed using
the branching process~$Z$ instead.  We complete our treatment of this
phase of development by proving two further lemmas.  The first shows that the 
branching approximation remains accurate until~$t = \t_{N,\a}^x$, defined
in~\Ref{AB-tau-x-def}.
The second shows that~$x_N(\t_{N,\a}^x)$ is close to a point on the solution~$\xi_N$
of~\Ref{deterministic} starting from~$x_{N,0}$,
except on an event $\hE_N$ whose complement has asymptotically negligible probability.  
The proofs are given in the appendix, Section~\ref{3-extra}.

\begin{lem}\label{stopping-at-level}
For $\t_{N,\a}^Z$ defined in~\Ref{AB-tau-Z-def},
let $\n_{N,\a}^Z$ denote the number of jumps made by~$Z$ until time $\t_{N,\a}^Z$,
infinite if $\t_{N,\a}^Z =  \infty$.
Then, under the assumptions of Section~\ref{assumptions}, there are  constants~$k_0$ and~$\th_0$ such that
\[
    \pr^0[ k_0 N^{1-\a} < \n_{N,\a}^Z < \infty] \Le e^{-\th_0 N^{1-\a}}.
\]
\end{lem}

\medskip

\begin{lem}\label{phase-1-final}
Suppose that $1/3 < \a < 1/2$. Then there is
a $\g > 0$ and an event~$\hE_N$
satisfying $\lim_{N\to\infty} \pr^0[\hE_N^c \cap \{\t_{N,\a}^x < \infty\}] = 0$ such that,
on the event $\hE_N \cap \{\t_{N,\a}^x < \infty\}$, we have
\[
   |x_N\ut(\t_{N,\a}^x) - \xi_N\ut(t_{N,\a}^\xi)| \ =\ O(N^{-\a-\g}),
\]
and
\[
   |x_N\ui(\t_{N,\a}^x) - \xi_N\ui(t_{N,\a}^\xi)| \ =\ O(N^{-\a}\log N).
\]
\end{lem}

\medskip

We summarize the results of this section in the following proposition.
For use in the sections to come, we specialize to $\a = 5/12$.

\begin{prop}\label{phase-1-conc}
Suppose that $|x_N\ui(0) - x_0\ui| \le N^{-5/12}$ and that $Nx_N\ut(0) = Z_0$,
for some fixed~$Z_0$. Define $\t^Z(0)$ as in~\Ref{AB-tau-0-defs}, and
$\t_{N*}^Z$, $\t_{N*}^x$ and~$t_{N*}^\xi$ as in~\Ref{AB-tau-star-defs}.
Then, under the assumptions of Section~\ref{assumptions}, it is possible to couple
the paths of~$Nx_N\ut$ and of the branching process~$Z$ in such a way that,
except on an event of asymptotically negligible probability,
they are identical until time~$\min\{\t^Z(0),\t_{N*}^Z\}$,
when, in particular, $\t_{N*}^Z = \t_{N*}^x$.
Furthermore,
there is a $\g > 0$ and constants $\bk\ui$ and~$\bk\ut$ such that, if $\t_{N*}^x < \infty$,
\eqs
   |x_N\ui(\t_{N*}^x) - \xi_N\ui(t_{N*}^\xi)| &\le& \bk\ui N^{-5/12}\log N; \\
   |x_N\ut(\t_{N*}^x) - \xi_N\ut(t_{N*}^\xi)| &\le& \bk\ut N^{-5/12-\g},
\ens
and
\[
    \t_{N*}^x \Eq \b_0^{-1}\{(1-\a)\log N - \log W\} + O(N^{-7/48}) ,
\]
except on an event~$\hE_N^c$ of negligible probability, where~$\xi_N$ is the
solution to the deterministic equation starting with $\xi_N(0) = x_N(0)$.
\end{prop}

\section{Intermediate growth}\label{phase-2}
\setcounter{equation}{0}
In the previous section, it has been shown that, on $\{\t_{N*}^x < \infty\}\cap \hE_N$, the point~$x_N(\t_{N*}^x)$
is close to~$\xi_N(t_{N*}^\xi)$, where~$\xi_N$ is the solution to~\Ref{deterministic} with
initial condition $\xi_N(0) = x_{N,0}$, and $t_{N*}^\xi$ is a non-random time defined in~\Ref{AB-tau-star-defs}.
We now show that $x_N(\t_{N*}^x + t)$ stays uniformly close
to $\xi_N(t_{N*}^\xi+t)$ for all $0 \le t \le t_0(\d',\e_N)$, for a suitably chosen~$\d' > 0$,
not depending on~$N$; here, and throughout the section, we define
\eq\label{eps-N-def}
   \e_N \Def |x_N\ut(\t_{N*}^x)| \ \asymp\ N^{-5/12},
\en
with the last relation and the inequality $\e_N \ge N^{-5/12}/|\bv|$ justified in view of
the definition~\Ref{AB-tau-star-defs} of~$\t_{N*}^x$.
We also show that~$\d'$ can be chosen so that all the components of $\xi_N(t_{N*}^\xi + t_0(\d',\e_N))$
are bounded away fron zero.  Hence, after this time,
Kurtz~(1970, Theorem~(3.1)) can be used to continue the approximation
of~$x_N$ by~$\xi_N$ along a non-degenerate path, as stated in Theorem~\ref{combined}.

We start by using the Markov property to continue from~$\t_{N*}^x$.
Let $x_1  := x_N(\t_{N*}^x)$, that is 
$x_1^{(1)}=x^{(1)}_N(\t_{N*}^x)$, $x_1^{(2)}=x^{(2)}_N(\t_{N*}^x)$, 

and define $\tx_N(t) := x_N(\t_{N*}^x + t)$.
Note that, from Lemma~\ref{d-lem-1},
\eqa
  |\xi_N\ui(t_{N*}^\xi) - x_0\ui| &\le& k\ui\{|x_{N,0}\ui - x_0\ui| + (|Z_0|/\bv^T Z_0)N^{-5/12}\} \non\\
      &\le& \bk\uh N^{-5/12},  \label{starting}
\ena
with $\bk\uh := k\ui(1 + \max_{1\le i\le d}\{1/\bv_i\})$.
Then we can write
\eqs
  \tx_N\ui(t) &=& x_1\ui + \int_0^t \{A(\tx_N(u)) \tx_N\ut(u) + c(\tx_N\ui(u))\}\,du + \tm_{N}\ui(t); \\
  \tx_N\ut(t) &=& x_1\ut + \int_0^t B(\tx_N(u)) \tx_N\ut(u)\,du + \tm_{N}\ut(t),
\ens
or, using variation of constants,
\eqa
    \tx_N\ui(t) &=& x_0\ui + e^{Ct}(x_1\ui - x_0\ui) \non\\[1ex]
      &&\mbox{} + \int_0^t e^{C(t-u)}\{ A(\tx_N(u)) \tx_N\ut(u) + \tc(\tx_N\ui(u))\}\,du \non\\
      &&\quad\mbox{} + \tm_{N}\ui(t) + C\int_0^t e^{C(t-u)} \tm_N\ui(u)\,du;  \label{1}\\
      \tx_N\ut(t) &=& x_1\ut + \int_0^t B(\tx_N(u)) \tx_N\ut(u)\,du + \tm_{N}\ut(t)\non\\
        &=& e^{B_0t}x_1\ut + \int_0^t e^{B_0(t-u)} \{B(\tx_N(u) - B_0\} \tx_N\ut(u)\,du \non\\
         &&\quad\mbox{}  + \tm_{N}\ut(t) + B_0\int_0^t e^{B_0(t-u)} \tm_N\ui(u)\,du, \label{2}
\ena
where $\tm_N(t) := m_N(\t_{N*}^x + t) - m_N(\t_{N*}^x)$, and~$m_N$ is as in~\Ref{AB-mN-def}.
The deterministic counterparts of \Ref{1} and~\Ref{2}
have been given previously in \Ref{3a-d} and~\Ref{4-d}.  We first use the comparison between
these pairs of equations to show that~$\tx_N$ stays close to~$\txi_N$, where~$\txi_N$
solves~\Ref{deterministic} with $\txi_N(0) = \tx_N(0) = x_1$.  Afterwards, we can use
Lemma~\ref{d-lem-1} to show that~$\txi_N(\cdot)$ stays uniformly close to~$\xi_N(t_{N*}^\xi+\cdot)$
in the appropriate time interval,
and that, at the end of this interval, $\xi_N$ is away from the boundary.
In preparation for the next result,
taking~$\d_0$ as in Lemma~\ref{d-lem-1},
note that $0 < t_0(\d_0,\e_N) = \b_0^{-1}\log(\d_0/\e_N) \le k_1\log N$ for a suitable choice of~$k_1$
and for all~$N$ sufficiently large.

\begin{lem}\label{phase-2-lem-new}
 There exist $\d_1 > 0$ and an event~$E_N$ with $\pr[E_N^c] = O(N^{-r})$ for any~$r > 0$ such that,
on~$E_N$, for all~$\d \le \d_1$,
\eqs
    \sup_{0\le t\le t_0(\d_0,\e_N)}|\tx_N\ui(t) - \txi_N\ui(t)| &\le& k\ui(\d_1)  N^{-1/12 + \chi(\d)};\\
   \sup_{0\le t\le t_0(\d_0,\e_N)}\{|\tx_N\ut(t) - \txi_N\ut(t)|/|\txi_N\ut(t)|\}
                 &\le& k\ut(\d_1)  N^{-1/12 + \chi(\d)},
\ens
for some $ k\ui(\d_1), k\ut(\d_1)$ and $\chi(\d) > 0$, where~$\lim_{\d\to0}\chi(\d) = 0$.
\end{lem}

\begin{proof}
Define the (random) time
\eqs
 \t_N &:=& \inf\{t > 0\colon\, |\tx_N(t)- x_0| > \r_2\},
\ens
and let~$E_N$ denote the event
\eq \label{good-event}
 \Bigl\{\max_{J\in\jj}\sup_{0\le t\le t_0(\d_0,\e_N)\wedge\t_N}|N^{-1}P^J(NG_N^J(t)) - G_N^J(t)|
                                      \le 2N^{-1/2}\sqrt{k_1} (\log N)^{3/2} \Bigr\},
\en
for~$k_1$ as defined above.
Then, by Lemma~\ref{MG-control}, $\pr[E_N^c] = O(N^{-r})$,
for any $r > 0$, and, on the event~$E_N$,
\eq\label{mN-bnd-2}
   \sup_{0\le t\le t_0(\d_0,\e_N)\wedge\t_N}|\tm_N(t)| \Le  \h_{N1} \Def \h_{N1}(k_1)
           \Eq O(N^{-1/2}(\log N)^{3/2}).
\en
The remaining argument involves careful use of the Gronwall inequality on the event~$E_N$,
to translate the smallness of~$\sup_{0\le t\le t_0(\d_0,\e_N)\wedge\t_N}|\tm_N(t)|$ into
a corresponding closeness of $\tx_N$ and~$\txi_N$ over a large part of this time interval.  The
main difficulty is that the length of the interval tends to
infinity with~$N$.

Taking the difference of \Ref{1} and~\Ref{3a-d}, we find that, on~$E_N$,
\eqa
  \lefteqn{|\tx_N\ui(t) - \txi_N\ui(t)| \Le
           \int_0^t |e^{C(t-u)} \{A(\tx_N(u))\tx_N\ut(u) - A(\txi_N(u))\txi_N\ut(u)\}|\,du}
                  \non \\
   \mbox{} &&\phantom{XXXXX}\qquad + \int_0^t |e^{C(t-u)}\{\tc(\tx_N\ui(u)) - \tc(\txi_N\ui(u))\}| \,du
               + \h_{N1}'\label{1-component-diff} \\
   &\le& \|A\|_{\r_2}\cst_1 \int_0^t |\tx_N\ut(u)-\txi_N\ut(u)|\,du
   \non \\ &&
      +\ \cst_1\int_0^t  |\txi_N\ut(u)| \|DA\|_{\r_2} |\tx_N(u) - \txi_N(u)|\,du \non \\
    &&\ \mbox{}  + \cst_1 \kkkc  \int_0^t |\tx_N\ui(u) - \txi_N\ui(u)|
                 \{|\tx_N\ui(u) - \txi_N\ui(u)| + |\txi_N\ui(u) - x_0\ui|\}\,du + \h_{N1}',
    \non
\ena
for $0\le t\le t_0(\d_0,\e_N)\wedge\t_N$, where $\h_{N1}' := (1 + \cst_1\|C\|k_1\log N)\h_{N1}$.
Writing
\eqs
    d_N\ui(t) &:=& \sup_{0\le u\le t\wedge\t_N}|\tx_N\ui(u) - \txi_N\ui(u)|;\\
    d_N\ut(t) &:=& \sup_{0\le u\le t\wedge\t_N}e^{-\b_0 u}|\tx_N\ut(u) - \txi_N\ut(u)|,
\ens
it thus follows, for $t\le t_0(\d_0,\e_N)\wedge\t_N$ and on~$E_N$, that
\eqa
    d_N\ui(t) &\le&   \cst_1\int_0^t  e^{\b_0 u}d_N\ut(u) \|A\|_{\r_2}\,du \non\\
              &&\mbox{}    + \cst_1 \|DA\|_{\r_2}
         \int_0^t  e^{\b_0u} \D_N\ut(u)(d_N\ui(u) + e^{\b_0u}d_N\ut(u))\,du \non \\
              &&\mbox{}\quad + \cst_1{\kkkc }d_N\ui(t)\int_0^t \{d_N\ui(u) + |\txi_N\ui(u) - x_0\ui|\}\,du
                   + \h_{N1}'.  \label{1-component-diff-2a}
\ena 
Using Lemma~\ref{d-lem-1} to bound~$\D_N\ut(t) := \sup_{0\le u\le t}e^{-\b_0u}|\txi_N\ut(u)|$,
in which we can take $\e\ut = \e_N$ and $\e\ui = (\bk\ui + \bk\uh) N^{-5/12}\log N$ for~$\txi_N(0)$,
because of Proposition~\ref{phase-1-conc}, \Ref{eps-N-def} and~\Ref{starting}, this gives
\eqa
    d_N\ui(t) &\le& \cst_1\b_0^{-1} e^{\b_0 t}\{d_N\ut(t)\|A\|_{\r_2}
    + k\ut \e_N  \|DA\|_{\r_2} (d_N\ui(t) + e^{\b_0t}d_N\ut(t))\}\non\\
   &&\qquad\mbox{}  + \tfrac14 d_N\ui(t)  + \h_{N1}',
     \label{1-component-diff-2}
\ena
for all~$t$ such that
\eq\label{integral-bnd-2}
    \cst_1 \kkkc  \int_0^t d_N\ui(u)\,du \Le 1/8 \quad \mbox{and}\quad
    \cst_1 \kkkc  \int_0^t |\txi_N\ui(u) - x_0\ui|\,du  \Le 1/8.
\en
Observe also that, from Lemma~\ref{d-lem-1},
\[
   \cst_1 \kkkc  \int_0^t |\txi_N\ui(u) - x_0\ui|\,du \Le
           \cst_1 \kkkc  k\ui\b_0^{-1}\{\e\ui\log(1/\e_N) + \d\}
\]
for $t \le t_0(\d,\e_N)$ and for any $\d \le \d_0$, where~$\d_0$ is as in Lemma~\ref{d-lem-1}.
With the above choice of~$\e\ui$ and for any~$\d = \d'$ chosen small enough, smaller than~$\d_0$ if necessary,
the right hand side is smaller than~$1/8$ for all~$N$ large enough.

Now choose $0 < \d_1 \le \min\{\d_0,\d'\}$ such that $\cst_1 k\ut\d_1\b_0^{-1}\|DA\|_{\r_2} \le 1/4$
and $\d_1 \le \r_2/2$,
and consider $t \le t_0(\d_1,\e_N)$ such that~\Ref{integral-bnd-2} is satisfied, and also such that
\eq\label{cond-1-stoch}
     \max\{d_N\ui(t), e^{\b_0t}d_N\ut(t)\} \Le \d_1,
\en
for which, immediately,  $t \le \t_N$ and $e^{\b_0t}\e_N \le \d_1$.
Then it follows from~\Ref{1-component-diff-2} on~$E_N$ that, for such~$t$,
\eq\label{9}
   d_N\ui(t) \Le 2\cst_1\b_0^{-1}e^{\b_0 t}d_N\ut(t)\{\|A\|_{\r_2} + k\ut\d_1\|DA\|_{\r_2} \}
      + 2\h_{N1}'.
\en

We now take the difference of \Ref{2} and~\Ref{4-d}, 
from which it follows on~$E_N$ that, for~$t$ as above,\par
\vbox{
\eqa
   |\tx_N\ut(t) - \txi_N\ut(t)|   
    &\le&   \int_0^t |e^{B_0(t-u)}(B(\tx_N(u)) - B(\txi_N(u)))\tx_N\ut(u)|\,du \non\\
   &&\mbox{} + \int_0^t |e^{B_0(t-u)}(B(\txi_N(u)) - B_0)(\tx_N\ut(u) - \txi_N\ut(u))|\,du \non\\
    &&\quad\ \mbox{}\  + \Blm\tm_{N}\ut(t) + B_0\int_0^t e^{B_0(t-u)} \tm_N\ui(u)\,du,\Brm
                   \label{part-2-diff}
\ena
}
\nin
giving, with $\h_{N1}^* := (1 + \cst_2\|B_0\|/\b_0)\h_{N1}$, and from Lemma~\ref{d-lem-1}
and~\Ref{cond-1-stoch},
\eqa
 d_N\ut(t) &\le& \cst_2\int_0^t d_N\ut(u) \|B(\tx_N(u)) - B(\txi_N(u))\| \,du  \non \\ &&\mbox{}\quad
    + \cst_2\int_0^t k\ut\e_N \|B(\tx_N(u)) - B(\txi_N(u))\| \,du \non\\
  &&\mbox{} \qquad + \cst_2\int_0^t d_N\ut(u)\|B(\txi_N(u)) - B_0\|\,du +  \h_{N1}^* \label{2-component-diff} \\
  &\le& k_{2}\left\{(\d_1 + |x_{1N}\ui - x_0\ui|)\int_0^t d_N\ut(u)\,du \right. \label{previous-1}\\
  &&\phantom{XXXXXXXXXXXXXX}\left.   \mbox{}
      + \e_N \int_0^t d_N\ui(u)\,du \right\} + \h_{N1}^*, \non
\ena
for a suitable constant~$k_{2}$. From~\Ref{9}, we have
\[
   \int_0^t d_N\ui(u)\,du \Le k_{3}e^{\b_0 t}\int_0^t d_N\ut(u)\,du
            + 2 t \h_{N1}',
\]
and, substituting this into~\Ref{previous-1}, we obtain
\eq\label{dN-1-este}
   d_N\ut(t) \Le k_{4}(\d_1 + |x_{1N}\ui - x_0\ui|) \int_0^t d_N\ut(u)\,du
              + \h_{N1}^* + k_{5}\e_N\log(1/\e_N) \h_{N1}',
\en
for constants $k_{4}$, $k_{5}$ and for $t \le t_0(\d_1,\e_N)$. Gronwall's inequality now yields
\[
    d_N\ut(t) \Le k_{6} \h_{N1} \exp\{k_{4}t(\d_1 + |x_{1N}\ui - x_0\ui|)\},
\]
for suitable~$k_{6}$,
and, for~$t = t_0(\d_1,\e_N)$, the right hand side can be made to be of order $O(N^{-1/2 + \chi})$,
for any~$\chi > 0$, by choosing~$\d_1 = \d_1(\chi)$ small enough.  In particular,
choosing~$t =  t_0(\d_1(\chi),\e_N)$ and recalling~\Ref{eps-N-def}, we have
\eq\label{2-tilde-bnd}
   \sup_{0\le u\le t}|\tx_N\ut(u) - \txi_N\ut(u)| \Le e^{\b_0 t}d_N\ut(t) \Eq O(\e_N^{-1}N^{-1/2+\chi})
       \Eq O(N^{-1/12 + \chi}),
\en
on the event~$E_N$, and also, in view of~\Ref{matrix-bnds-2} and the third inequality in Lemma~\ref{d-lem-1},
\[
   \sup_{0\le u\le t}\{|\tx_N\ut(u) - \txi_N\ut(u)|/|\txi_N\ut(u)|\} \Le \e_N^{-1}d_N\ut(t)
       \Eq O(N^{-1/12 + \chi}).
\]
In addition, from~\Ref{9} and~\Ref{2-tilde-bnd}, it follows on~$E_N$ that
\eq\label{AB-1-tilde-bnd}
    \sup_{0\le u\le t}|\tx_N\ui(u) - \txi_N\ui(u)| \ =:\
  d_N\ui(t) \Le k_{7}\{e^{\b_0 t}d_N\ut(t) + \h_{N1}'\}  \Eq O(N^{-1/12 + \chi}).
\en
We now compare the assumed conditions \Ref{integral-bnd-2} and~\Ref{cond-1-stoch},
involving bounds on increasing processes with jumps bounded by $\DDD = N^{-1}\max_{J\in\jj}|J|$,
with the resulting estimates \Ref{2-tilde-bnd} and~\Ref{AB-1-tilde-bnd}.  It then
follows immediately from Lemma~\ref{AB-little-lemma} that both \Ref{2-tilde-bnd} and~\Ref{AB-1-tilde-bnd}
hold on~$E_N$ for all $t \le  t_0(\d_1(\chi),\e_N)$ and for all~$N$ sufficiently large, provided
that $\chi < 1/12$.
\end{proof}

\medskip
It remains first to observe that the solution~$\txi_N$ of the deterministic equations
starting from $\txi_0 := x_1 = x_N(\t_{N*}^x)$ is close to the solution~$\hxi_N$ starting
from~$\hxi_N(0) = \xi_N(t_{N*}^\xi)$, up to the time $t_0(\d_1,\e_N)$,
because their starting points are close enough
on~$\hE_N$, as was shown in Proposition~\ref{phase-1-conc};
from the final statements of Lemma~\ref{d-lem-1}, taking $\e\ui = \bk\ui N^{-5/12}\log N$ and
$\e\ut = \bk\ut N^{-5/12}$, it follows that, for some $\g > 0$,
\eqa
    \sup_{0\le u\le t_0(\d_1,\e\ut)} |\hxi_N\ui(u) - \txi_N\ui(u)| &=& O(N^{-\g});\non\\
    \sup_{0\le u\le t_0(\d_1,\e\ut)} |\hxi_N\ut(u) - \txi_N\ut(u)| &=& O(N^{-\g}).
          \label{AB-close-at-end}
\ena

One final result is needed, to show that continuation using Kurtz~(1970, Theorem~(3.1)) 
represents following the deterministic path along an asymptotically non-degenerate path.
The proof is given in Section~\ref{4-extra}.

\begin{lem}\label{non-degenerate}
 Define
\[
    \hht_N \Def t_{N*}^\xi + t_0(\d',\e_N) \Eq \b_0^{-1}\{\log N + \log(\d'/\bv^T Z_0)\}.
\]
Then, for suitably chosen $\d' \le \d_1$, all the components of $\txi_N(t_0(\d',\e\ut)) = \xi_N(\hht_N)$
are uniformly bounded away from zero for all~$N$ large enough.
\end{lem}

\medskip
We summarize the results of this section in the following proposition, which, 
with Proposition~\ref{phase-1-conc}, completes the proof of Theorem~\ref{combined}.
Theorem~\ref{combined-2} is proved in Section~\ref{phase-3} in the appendix.

\begin{prop}\label{phase-2-conc}
Let $\xi_N$ denote the
solution to the deterministic equation starting with $\xi_N(0) = x_{N,0}$ satisfying
$|x_{N,0}\ui - x_0\ui| \le N^{-5/12}$ and $x_{N,0}\ut = N^{-1}Z_0$.
Let $\e_N \asymp N^{-5/12}$ be as defined in~\Ref{eps-N-def}, $t_0(\d,\e)$ as in~\Ref{t-delta-eps-def},
and $\t_{N*}^x$ and~$t_{N*}^\xi$ as in~\Ref{AB-tau-star-defs}.
Then there exist $\d' > 0$ and an event~$E_N$, whose complement has asymptotically negligible
probability, such that, on~$E_N \cap \{\t_{N*}^x < \infty\}$ and for all~$N$ large enough,
\[
    \sup_{0\le t\le t_0(\d',\e_N)}|x_N(\t_{N*}^x + t) - \xi_N(t_{N*}^\xi + t)| \Le k(\d')  N^{-\g},
\]
for some $\g > 0$ and $0 < k(\d') < \infty$, and that all components of
$\xi_N(t_{N*}^\xi + t_0(\d',\e_N))$ are bounded uniformly away from zero.
Note also that
\eqs
    t_{N*}^\xi + t_0(\d',\e_N) &=& \b_0^{-1}\{\tfrac7{12}\log N - \log(\bv^T Z_0) + \log \d' - \log\e_N\} \\
     &=& \b_0^{-1}\log N + O(1).
\ens
\end{prop}

\acks
ADB was supported in part by Australian Research Council Grants Nos DP120102728 and DP120102398,
KH and FK by Australian Research Council Grant No~DP120102728, and
HK by the Israel Science Foundation Grant 764/13 and by the Milford Bohm Memorial Grant.
The authors wish to warmly thank the referees of an earlier version, whose contributions have
substantially improved the paper.

\appendix


\section{The approach to extinction}\label{phase-3}
\setcounter{equation}{0}
In this section, we briefly examine the situation discussed in Section~\ref{absorption},
when the equilibrium at~$x_0 = [x_0\ui,0]$ is strongly stable, with~\Ref{matrix-bnd-attract}
satisfied, and when the stochastic process~$x_N$
starts not too far from~$x_0$; in such circumstances, the second component~$x_N\ut$ of~$x_N$
goes to zero.
We first suppose that the stochastic and deterministic processes $x_{N,\d}$ and~$\xi_\d$
start with $\xi_\d(0) = x_{N,\d}(0) = x_{\d0}$, where $|x_{\d0} - x_0| \le \d$. For
$0 < \d < 1$ chosen small enough, we then show that they
remain close for a further time $\b_1^{-1}(\log\d + \tfrac5{12}\log N)$, at which point
the second coordinates are of magnitude approximately~$N^{-5/12}$ and the first coordinates
are at a similar distance from~$x_0\ui$.

Define $\e_N := N^{-5/12}$,
and note that, for $\d \le 1$, $t_1(\d,\e_N) \le t_1(1,\e_N) = k_1'\log N$ with $k_1' := 5/(12\b_1)$
where~$t_1$ is as in~\Ref{t-delta-eps-def}. Let
$x_{N,\d}$ denote the stochastic process starting at~$x_{\d0}$, chosen as above,
for $0 < \d \le \d_0$, where~$\d_0$ is as in Lemma~\ref{d-lem-2}.

\begin{lem}\label{phase-2-lem-new*}
 There exist $\d_1 > 0$ and an event~$E_N$ with $\pr[E_N^c] = O(N^{-r})$ for any~$r > 0$ such that,
on~$E_N$, for all~$\d \le \d_1$,
\eqs
    \sup_{0\le t\le t_1(\d,\e_N)}|x_{N,\d}\ui(t) - \xi_\d\ui(t)| &\le& k\ui(\d)  N^{-1/12 + \chi(\d)};\\
   \sup_{0\le t\le t_1(\d,\e_N)}\{e^{\b_1t}|x_{N,\d}\ut(t) - \xi_\d\ut(t)|\}
                 &\le& k\ut(\d)  N^{-1/12 + \chi(\d)},
\ens
for some $ k\ui(\d), k\ut(\d)$ and $\chi(\d) > 0$, where~$\lim_{\d\to0}\chi(\d) = 0$.
\end{lem}

\begin{proof}
The argument is very like that for Lemma~\ref{phase-2-lem-new}, modified to take into
account that the trajectories converge towards the deterministic equilibrium at $[x_0,0]$.
As before, define the (random) time
\eqs
 \t_N &:=& \inf\{t > 0\colon\, |x_{N,\d}(t)- x_0| > \r_2\},
\ens
and let~$E_N$ denote the event
\eq\label{good-event*}
    E_N \Def \Bigl\{\max_{J\in\jj}\sup_{0\le t\le t_1(\d_0,\e_N)\wedge\t_N}|N^{-1}P^J(NG_N^J(t)) - G_N^J(t)|
                                      \le 2N^{-1/2}\sqrt{k_1'} (\log N)^{3/2} \Bigr\}.
\en
Then, as in Lemma~\ref{MG-control}, $\pr[E_N^c] = O(N^{-r})$,
for any $r > 0$, and, on the event~$E_N$,
\eq\label{mN-bnd-2*}
   \sup_{0\le t\le t_1(\d_0,\e_N)\wedge\t_N}|m_N(t)| \Le  \h_{N1} \Def \h_{N1}(k_1').
\en

We now use equations \Ref{1}, \Ref{2}, \Ref{3a-d} and~\Ref{4-d} much as in the proof of
Lemma~\ref{phase-2-lem-new}, but with $x_{N,\d}$ and $\xi_\d$ in place
of $\tx_N$ and~$\xi$.
Taking the difference of \Ref{1} and~\Ref{3a-d}, and recalling~\Ref{matrix-bnd-attract}, we find that, on~$E_N$,
\eqa
  \lefteqn{|x_{N,\d}\ui(t) - \xi_\d\ui(t)| \Le
           \int_0^t |e^{C(t-u)} \{A(x_{N,\d}(u))x_{N,\d}\ut(u) - A(\xi_\d(u))\xi_\d\ut(u)\}|\,du}
                  \non \\
   \mbox{} &&\phantom{XXXXXX}\qquad + \int_0^t |e^{C(t-u)}\{\tc(x_{N,\d}\ui(u)) - \tc(\xi_\d\ui(u))\}| \,du
               + \h_{N1}'\phantom{XXXX}\label{1-component-diff*} \\
   &\le& \|A\|_{\r_2}\cst_1 \int_0^t e^{-\k(t-u)}|x_{N,\d}\ut(u)-\xi_\d\ut(u)|\,du \non \\
   &&\mbox{}+\cst_1 \int_0^t e^{-\k(t-u)} |\xi_\d\ut(u)| \|DA\|_{\r_2} |x_{N,\d}(u) - \xi_\d(u)|\,du + \h_{N1}'
                  \label{expanded-bnd} \\
    &&\ \mbox{} + \kkkc \cst_1 \int_0^t e^{-\k(t-u)} |x_{N,\d}\ui(u) - \xi_\d\ui(u)|
              \{|x_{N,\d}\ui(u) - \xi_\d\ui(u)| + |\xi_\d\ui(u) - x_0\ui|\}\,du ,
   \non
\ena
for $0\le t\le t_1(\d_0,\e_N)\wedge\t_N$, where $\h_{N1}' := (1 + \cst_1 \|C\|/\k)\h_{N1}$.
Write $d_N\ui(t) := \sup_{0\le u\le t\wedge\t_N} e^{\k'u}|x_{N,\d}\ui(u) - \xi_\d\ui(u)|$
for some $\k' < \min\{\b_1,\k\}$,
and then write $d_N\ut(t) := \sup_{0\le u\le t\wedge\t_N}e^{\b_1 u}|x_{N,\d}\ut(u) - \xi_\d\ut(u)|$;
set $\d\ut(t) := \sup_{0\le u\le t}e^{\b_1 u}|\xi_\d\ut(u)|$. Then,
for any $\d \le \d_0$ and on~$E_N$, it follows that
\eqa
    \lefteqn{e^{\k't}|x_{N,\d}\ui(t) - \xi_\d\ui(t)|}\non\\
     &\le&  \cst_1 e^{\k't} \int_0^t e^{-\k(t-u)} e^{-\b_1 u}d_N\ut(u)
         \|A\|_{\r_2} 
             \,du \non\\
              &&\mbox{}    + \cst_1 \|DA\|_{\r_2}e^{\k't}
         \int_0^t e^{-\k(t-u)}  e^{-\b_1u} \d\ut(u)\{e^{-\k'u}d_N\ui(u) + e^{-\b_1u}d_N\ut(u)\}\,du \non \\
     &&\mbox{}\quad + \cst_1 {\kkkc } e^{\k't} d_N\ui(t) \int_0^t e^{-\k(t-u)-\k'u}
                   \{ e^{-\k'u}d_N\ui(u) + |\xi_\d\ui(u) - x_0\ui| \}\,du + e^{\k't}\h_{N1}'
            \non \\[1ex]
  &\le& \cst_1 \b_1^{-1} d_N\ut(t)\|A\|_{\r_2}
          \non\\[1ex]
   && +\ \cst_1 k\ut\d \b_1^{-1}  \|DA\|_{\r_2} (d_N\ui(t) + d_N\ut(t))
     + \tfrac14 d_N\ui(t)  + e^{\k't}\h_{N1}',
     \label{1-component-diff-2*}
\ena
for $t\le t_1(\d,\e_N)\wedge\t_N$ such that
\eq\label{integral-cond-3}
   \cst_1 \kkkc  \int_0^t  d_N\ui(u)\,du \Le 1/8 \quad\mbox{and}\quad
     \cst_1 \kkkc  \int_0^t |\xi_\d\ui(u) - x_0\ui|\,du \Le 1/8;
\en
we have also used Lemma~\ref{d-lem-2} to bound~$\d\ut(t)$.
Note that Lemma~\ref{d-lem-2} gives the bound
$\cst_1 \kkkc  \hk_1 (\k')^{-1}\d$ for the second of the integrals in~\Ref{integral-cond-3}, and
this is smaller than~$1/8$ for all~$\d\le \d'$ chosen small enough.
Now choose $\d_1 \le \min\{\d_0,\d'\}$ such that $\d_1 \le \r_2/2$ and also
\hbox{$\cst_1 k\ut \d_1 \b_1^{-1}\|DA\|_{\r_2} \le 1/4$,}
and consider $t \le t_N(\d_1)$ such that also
\eq\label{cond-1-stoch*}
     \max\{d_N\ui(t),d_N\ut(t)\} \Le \d_1,
\en
implying $t \le \t_N$ also.  Then it follows on~$E_N$ that, for such~$t$,
\eq\label{9*}
   d_N\ui(t) \Le 2\cst_1 \b_1^{-1} d_N\ut(t)\{\|A\|_{\r_2} + k\ut\d_1\|DA\|_{\r_2} \}
      + 2e^{\k't}\h_{N1}'.
\en

We now, as for~\Ref{part-2-diff}, take the difference of \Ref{2} and~\Ref{4-d}, from which it follows on~$E_N$
that, for~$t$ as above and for $\d \le \d_1$,
\eqa
  \lefteqn{|x_{N,\d}\ut(t) - \xi_\d\ut(t)|} \non\\
    &&\Le   \int_0^t |e^{B_0(t-u)}(B(x_{N,\d}(u)) - B(\xi_\d(u)))x_{N,\d}\ut(u)|\,du \non\\
   &&\mbox{}\quad\ + \int_0^t |e^{B_0(t-u)}(B(\xi_\d(u)) - B_0)(x_{N,\d}\ut(u) - \xi_\d\ut(u))|\,du \non\\
    &&\quad\ \quad\mbox{}\  + \Blm\tm_{N}\ut(t) + B_0\int_0^t e^{B_0(t-u)} \tm_N\ui(u)\,du \Brm.
                   \label{part-2-diff*}
\ena
Defining $\cst_2'$ so as to satisfy $|e^{B_0 t}x| \le \cst_2'e^{-\b_1t}|x|$ for all $t\ge0$,
$x\in \reals^{d_2}$,  and writing $\h_{N1}^* := (1 + \cst_2'\|B_0\|/\b_1)\h_{N1}$,
\Ref{part-2-diff*} together with Lemma~\ref{d-lem-2} gives
\eqa
 d_N\ut(t) &\le& \cst_2'\int_0^t d_N\ut(u) \|B(x_{N,\d}(u)) - B(\xi_\d(u))\| \,du  \non \\ &&\mbox{}\quad
    + \cst_2'\int_0^t \d\ut(u) \|B(x_{N,\d}(u)) - B(\xi_\d(u))\| \,du \non\\
  &&\mbox{} \qquad + \cst_2'\int_0^t d_N\ut(u)\|B(\xi_\d(u)) - B_0\|\,du +   e^{\b_1t}\h_{N1}^*
             \label{2-component-diff*} \\
  &\le& k'_{2}\d_1\Blb \int_0^t d_N\ut(u)\,du
      +  \int_0^t d_N\ui(u)\,du \Brb + e^{\b_1t}\h_{N1}^*, \label{previous-1*}
\ena
this last from~\Ref{cond-1-stoch*} and for a suitable constant~$k'_{2}$. From~\Ref{9*}, we have
\[
   \int_0^t d_N\ui(u)\,du \Le k'_{3}\int_0^t d_N\ut(u)\,du
            + 2te^{\k't}  \h_{N1}',
\]
and, combining this with~\Ref{previous-1*}, we obtain
\eq\label{dN-1-este*}
   d_N\ut(t) \Le k'_{4}\d_1 \int_0^t d_N\ut(u)\,du
              + e^{\b_1t}\h_{N1}^* + k'_{5}t e^{\k't}\h_{N1}',
\en
for constants $k'_{4}$, $k'_{5}$. Since $\k' < \b_1$, Gronwall's inequality now yields
\eq\label{GR}
    d_N\ut(t) \Le k'_{6} e^{\b_1t}\h_{N1} \exp\{k'_{4}t\d_1 \},
\en
for suitable~$k'_{6}$,
and, for~$t = t_1(\d_1,\e_N)$, the right hand side can be made to be of order $O(N^{-1/12 + \chi})$,
for any~$\chi > 0$, by choosing~$\d_1 = \d_1(\chi)$ small enough.  In particular,
choosing~$t = t_1(\d_1(\chi),\e_N)$, we have
\eq\label{GR-2}
   \sup_{0\le u\le t}e^{\b_1u}|x_{N,\d}\ut(u) - \xi_\d\ut(u)| \Eq d_N\ut(t)
       \Eq O(N^{-1/12 + \chi}),
\en
on the event~$E_N$.  In addition, from~\Ref{9*}, it follows on~$E_N$ that
\eq\label{AB-GR-3}
    \sup_{0\le u\le t}|x_{N,\d}\ui(u) - \xi_\d\ui(u)|
     \Le k'_{7}\{d_N\ut(t) + e^{\k't}\h_{N1}'\}  \Eq O(N^{-1/12 + \chi}),
\en
again because $\k' < \b_1$, and in view of~\Ref{GR-2}.
Now the assumed conditions \Ref{integral-cond-3} and~\Ref{cond-1-stoch*} bound functions
having jumps no larger than $\DDD = \d N^{-7/12}\max_{J\in\jj}|J|$, and, comparing them
with the bounds derived from \Ref{GR-2} and~\Ref{AB-GR-3}, we can invoke
Lemma~\ref{AB-little-lemma}
to show that, on~$E_N$ and for all~$N$ sufficiently large,
the latter bounds are satisfied for all $t \le t_1(\d_1(\chi),\e_N)$,
provided that $\chi < 1/12$.
\end{proof}

\medskip
In view of Lemma~\ref{phase-2-lem-new*}, and provided that~$\d$ is chosen small enough,
the process~$x_{N,\d}$, started from a position~$x_{\d0}$ within
distance~$\d$ of its equilibrium $[x_0,0]$, remains asymptotically close to the
deterministic curve, starting from the same initial conditions, until time~$t = t_1(\d,\e_N)$,
at which point $|x_{N,\d}\ut(t)| = O(N^{-5/12})$.  Instead, we might think of comparing~$x_{N,\d}$ with
a deterministic solution~$\txi_\d$ with a slightly different initial condition,
as, for instance, the terminal value of the deterministic solution used in Theorem~\ref{combined}.
So suppose that~$\txi_\d(0)$ satisfies $|\xi_\d(0) - \txi_\d(0)| = \e\uf$,
where $|\e\uf| = O(N^{-\hhh_1})$ for some $\hhh_1 > 0$; the choice $\hhh_1 = 1/24$
would correspond to the initial condition from Theorem~\ref{combined}.
Then we can use the last part of Lemma~\ref{d-lem-2} to bound the
extra error in approximating~$x_{N,\d}$ by~$\txi_\d$.  We are interested in taking
$\h \asymp \e_N = N^{-5/12}$ in the lemma, so as to reach time~$t_1(\d,\e_N)$;
for this choice of~$\h$, the condition $\e\uf\h^{-\th} \le \KKK$ in Lemma~\ref{d-lem-2} can
be realized for all~$N$ large enough, by taking~$\th := 6\hhh_1/5$. This gives
\eqa\label{extra-error}
   &&|\xi_\d\ui(t_1(\d,\e_N)) - \txi_\d\ui(t_1(\d,\e_N))| \Eq O(N^{-\hhh_1/2});\phantom{XXXXXX}\non\\
   &&e^{\b_1t_N(\d)}|\xi_\d\ut(t_1(\d,\e_N)) - \txi_\d\ut(t_1(\d,\e_N))| \Eq O(N^{-\hhh_1/2}).
\ena
Thus the overall error in the approximation may be a little
worse than that proved in Lemma~\ref{phase-2-lem-new*}, but is of the same general form.

Thereafter, $Nx_N\ut$ can be well approximated
by the branching process~$Z$, with asymptotically small error in total variation,
until (if it does so) it has made $kN^{7/12}$ further transitions, for any fixed $k > 0$.
This can be proved by an argument almost identical to that in Section~\ref{phase-1},
and we do not repeat it.
By choosing~$k$ large enough, the probability that the branching process does not become
extinct before it has made $kN^{7/12}$ further transitions is asymptotically small.
The branching process is itself very easy to describe, since it is the sum of $cN^{7/12}$
independent paths, each of which represents the (a.s.\ finite) family tree descended
from one of the $N\bone^Tx_N\ut(t_N(\d))$ initial individuals.  These family trees have
distributions that are identical for ancestors of the same type, and each has finite mean
and variance of the number of individuals ever born.  Hence classical limit theory can
be used to approximate the remaining behaviour very well.

Perhaps the most interesting feature of this final phase is the time until ultimate extinction.
For the Markov branching process, it is shown in Heinzmann~(2009) that its distribution has a Gumbel
limit, after centring, with an error that is small as the number of initial individuals tends to infinity.
Here, this is the case, since the initial number of individuals tends to infinity like~$N^{7/12}$.
Putting these facts together yields Theorem~\ref{combined-2} of Section~\ref{absorption}.



\section{Proofs of lemmas from Section~\ref{phase-1}}\label{3-extra}
\setcounter{equation}{0}
In view of Lemma~\ref{phase-1-lem}, probability calculations
can be conducted using~$Z$, with asymptotically negligible error.

\medskip
\nin{\bf Proof of Lemma~\ref{stopping-at-level}.}\
For each $d_1 < i \le d$, we have
\eq\label{phi-def}
  \f_i(\th) \Def \ex\Bigl\{ e^{-\th(\bv^T Z_1 - \bv_i)} \Giv Z_0 = \e^i \Bigr\} \ <\ 1,
\en
for all~$\th$ sufficiently small, where~$\e^i$ denotes the $i$-th coordinate vector.
Choose $\th_0 > 0$ such that~\Ref{phi-def} is satisfied for all~$i$, and write
\[
   \th_1 \Def - \log\Bigl\{ \max_{d_1 < i \le d} \f_i(\th_0) \Bigr\} \ >\ 0.
\]
Then it is immediate that, for $n\ge1$,
\[
    \ex\Bigl\{ e^{-\th_0(\bv^T Z_n - \bv^T Z_0)}I_{[\bone^T Z_n > 0]} \Giv Z_0 \Bigr\} \Le e^{-\th_1 n}
\]
and thus, for $k_0 := (2\th_0/\th_1)$ and $n = k_0 N^{1-\a}$,
\eqs
  \phantom{XX}\lefteqn{\pr[\{\n_{N,\a}^Z < \infty\} \setminus \{\n_{N,\a}^Z \le n\} \giv Z_0]}\phantom{XXXX} \\
   &=& \pr[\{\bv^T Z_n - \bv^T Z_0 < N^{1-\a} \} \cap \{\bone^T Z_n > 0\} \giv Z_0 \}] \\
          &\le& \exp\{\th_0 N^{1-\a} - \th_1 n\} \Eq \exp\{-\th_0 N^{1-\a}\}. \hskip2.9cm\proofbox
\ens

\bigskip\nin{\bf Proof of Lemma~\ref{phase-1-final}.}\
We first note two inequalities satisfied by the Markov branching process~$Z$.
They are continuous time analogues of results which,
for a square integrable multitype branching process in discrete time, follow
directly from Harris~(1951, (3.11)), and we prove them below, 
using the martingale representation of a branching process given 
in Klebaner ((1994), (1.4)).  The first delimits the rate of convergence
of the martingale~$W^\bv(t) := \bv^T Z(t) e^{-\b_0t}$ to its limit~$W$, and the second
the rate of convergence of $Z(t) e^{-\b_0t}$ to its limit~$W\bu$.  Defining the events
\eqs
   E_1(t,a) &:=& \Bigl\{ \sup_{u\ge t}|\bv^T Z(u) e^{-\b_0u} - W| \le ae^{-\b_0t/2} \Bigr\};\\
   E_2(t,K,\ch) &:=& \Bigl\{ \sup_{u\ge t}|Z(u) e^{-\b_0u} - W\bu| \le  Ke^{-\ch t} \Bigr\},
\ens
there exist constants~$k,K,\ch_1$ and~$\ch_2$ such that, for any $a,t > 0$,
\eq\label{branching-claim}
  \pr[\{E_1(t,a)\}^c] \Le  ka^{-2} \quad\mbox{and}\quad \pr[\{E_2(t,K,\ch_1)\}^c\cap \{W > 0\}] \Le e^{-\ch_2 t}.
\en

Fix $\e > 0$, to be chosen below, and, with $t_{N,\a,\e} := \b_0^{-1}(1-\a-\e)\log N$, define
\[
    E_{N,\a}^\e \Def E_{N}^{\e,0} \cap E_1(t_{N,\a,\e},\half N^\e),
\]
where $E_{N}^{\e,0} := \{\half N^{-\e} < W < \half N^\e\}$. Then, on~$E_{N,\a}^\e$ and for $t\le t_{N,\a,\e}$, 
we have
\[
    \bv^T Z(t) \ <\ N^{1-\a-\e}( \half N^\e + \half N^\e ) \ <\ N^{1-\a}.
\]
This implies that $\t_{N,\a}^Z > t_{N,\a,\e}$, and hence, from the definitions of $\t_{N,\a}^Z$ and~$E_{N,\a}^\e$,
that
\eq\label{AB-tau-approx}
  |N^{1-\a} e^{-\b_0\t_{N,\a}^Z} W^{-1} - 1| \Le   N^\e W^{-1}\,N^{-(1-\a-\e)/2}
       \Le N^{-(1-\a)/4},
\en
on~$E_{N,\a}^\e$, for all~$N$ large enough, with the final inequality true if~$\e$ is chosen to be
less than $(1-\a)/10$.  By opening it up and taking logarithms, it then follows by using $\log(1+x)\le x$, that 
\eq\label{AB-tau-approx-2}
   |\t_{N,\a}^Z - \b_0^{-1}\{(1-\a)\log N - \log W\}| \Le 2\b_0^{-1}N^{-(1-\a)/4} 
\en
on~$E_{N,\a}^\e$, for all~$N$ large enough, and hence also that
\eq\label{AB-tau-bnd}
   e^{\b_0\t_{N,\a}^Z} \Le 2W^{-1}N^{1-\a}.
\en

Now, on~$E_{N,\a}^{\e,2} := E_{N,\a}^\e \cap E_2(t_{N,\a,\e},K,\ch_1)$, and from the definition of~$t_{N,\a,\e}$,
we also have
\[
    |e^{-\b_0\t_{N,\a}^Z} Z(\t_{N,\a}^Z) - W\bu| \le KN^{-\g_1},
\]
where $\g_1 := (1-\a-\e)\ch_1/\b_0$, implying from~\Ref{AB-tau-bnd} and the definition of~$E_{N,\a}^\e$
that 
\begin{equation}\label{gamma1}
    |N^{-1}Z(\t_{N,\a}^Z) - N^{-1} W e^{\b_0\t_{N,\a}^Z}\bu| \Le KN^{-1-\g_1} e^{\b_0\t_{N,\a}^Z}
           \Le 2Ke N^{-\g_1+\e}\,N^{-\a}.
\end{equation}
Then, from first inequality in~\Ref{AB-tau-approx}, we also have on~$E_{N,\a}^\e$
\[
    |N^{-1} W e^{\b_0\t_{N,\a}^Z} - N^{-\a}| \Le 3N^{-\a}\,N^{-(1-\a)/4},
\]
for all~$N$ large enough.  Thus, on~$E_{N,\a}^{\e,2}$ and for all~$N$ large
enough, we have
\eq\label{AB-Z-approx}
  |N^{-1}Z(\t_{N,\a}^Z) - N^{-\a}\bu| \Le K' N^{-\a-\g}
\en
for some $K'<\infty$ and for some $\g = \g^\e > 0$, provided that $0 < \e < \min\{1,9\ch_1/\b_0\}(1-\a)/10$, 
to make $-\gamma_1+\e-\alpha<0$ in \eqref{gamma1}.

We now note that
\eqs
    \lefteqn{(E_{N,\a}^{\e,2})^c \cap \{\t_{N,\a}^Z < \infty\}}\\
 && \ \subset\ (E_1(t_{N,\a,\e},\half N^\e))^c
       \cup (E_2(t_{N,\a,\e},K,\ch_1))^c
          \cup  (\{\t_{N,\a}^Z < \infty\} \setminus E_{N}^{\e,0}).
\ens
The probabilities of the first two events in the union are asymptotically negligible,
from~\Ref{branching-claim}, and so are both $\pr[0 < W < \half N^{-\e}]$
and $\pr[W >  \half N^\e]$, since~$W$ is a proper random variable.  Then,
by elementary branching process theory (Athreya \&~Ney 1972, Chapter V.3, Theorem~2(i)),
\[
   \pr[W = 0 \giv \t_{N,\a}^Z < \infty] \Le \{\max_i q_i\}^{N^{1-\a}},
\]
where $q_i := \pr_i[\t^Z(0) < \infty] < 1$ is the extinction probability starting from
a single individual of type~$i$, and so $\lim_{N\to\infty}\pr[W = 0 \giv \t_{N,\a}^Z < \infty] = 0$.
Hence
\eqs
    \lefteqn{\{\t_{N,\a}^Z < \infty\} \setminus E_{N}^{\e,0} } \\
  &&\Eq \{0 < W < \half N^{-\e}\} \cup \{W >  \half N^\e\} \cup \bigl\{\{W=0\}\cap\{\t_{N,\a}^Z < \infty\}\bigr\}
\ens
also has asymptotically negligible probability.
Finally, from Lemmas \ref{phase-1-lem} and~\ref{stopping-at-level}, we can couple $Z$ and~$x_N\ut$ in
such a way that
\[
   \lim_{N\to\infty}\pr[\{\t_{N,\a}^Z \neq \t_{N,\a}^x\} \cup \{Z(\t_{N,\a}^Z) \neq Nx_N\ut(\t_{N,\a}^x)\}]
          \Eq 0.
\]
Hence we conclude that there is a $\g>0$ such that, for
\eq\label{x2-bnd}
   \hE_N \Def \{ |x_N\ut(\t_{N,\a}^x) - N^{-\a}\bu| \le N^{-\a-\g} \} \cap E_n(k),
\en
the event $\hE_N^c \cap \{\t_{N,\a}^x < \infty\}$ has asymptotically negligible probability, and that,
on this event,
\eq\label{AB-tau-N-alpha-asymp}
   \t_{N,\a}^x \Eq \b_0^{-1}\{(1-\a)\log N - \log W\} + O(N^{-(1-\a)/4}) ,
\en
thanks to~\Ref{AB-tau-approx-2}.
For the components~$x_N\ui(\t_{N,\a}^x)$,
it then follows from Lemma~\ref{close-start} that, on~$\hE_N$ and if $\a > 1/2$,
\eq\label{x1-bnd}
    |x_N\ui(\t_{N,\a}^x) - x_0\ui| \Eq O(N^{-\a}\log N).
\en

For the corresponding deterministic values, we use Lemma~\ref{d-lem-1}.
We take $\e\ut = N^{-1}|Z_0|$ and $\e\ui = O(N^{-\a})$, so that the
conditions are satisfied for all~$N$ sufficiently large;  noting that
$e^{\b_0 t_{N,\a}^\xi} = O(N^{1-\a})$, we conclude that
\eq
   |\xi_N\ui(t_{N,\a}^\xi) - x_{N,0}\ui| \Eq O(N^{-\a}); \quad
           |\xi_N\ut(t_{N,\a}^\xi) - e^{B_0 t_{N,\a}^\xi} x_{N,0}\ut| \Eq O(N^{-2\a}).
      \label{z2-bnd}
\en
On the other hand, by the Perron--Frobenius theorem,
\[
    |e^{B_0t}x_{N,0}\ut - e^{\b_0 t} \bv^T x_{N,0}\ut \bu| \Eq O(|x_{N,0}\ut|e^{(\b_0 - \g)t}),
\]
for some~$\g > 0$,  uniformly in $t > 0$, implying in turn from~\Ref{z2-bnd} that
\eqa
   |\xi_N\ut(t_{N,\a}^\xi) - e^{\b_0 t_{N,\a}^\xi} \bv^T x_{N,0}\ut \bu| &=& |\xi_N\ut(t_{N,\a}^\xi) - N^{-\a} \bu|
                   \phantom{XXXXXXXXXX}\non\\
        &=& O(N^{-\a}e^{-\g t_{N,\a}^\xi}) \Eq O(N^{-\a-\g'}), \label{det-proportions}
\ena
for some $\g' > 0$.  From this, and in view of~\Ref{x2-bnd}, it follows that, on the event~$\hE_N$,
\[
   |x_N\ut(\t_{N,\a}^x) - \xi_N\ut(t_{N,\a}^\xi)| \ =\ O(N^{-\a-\g}),
\]
for some $\g > 0$. 
The analogous bound from \Ref{x1-bnd} and~\Ref{z2-bnd}, completing the proof, is
\[
   \phantom{XXXXXXXXX}|x_N\ui(\t_{N,\a}^x) - \xi_N\ui(t_{N,\a}^\xi)| \ =\ O(N^{-\a}\log N).
       \phantom{XXXXXXXX}\proofbox
\]

\medskip
\nin{\bf Proof of~\Ref{branching-claim}.}\
Let $Z(t)=(Z^1(t),\ldots,Z^d(t))$ be a multitype Markov branching process, in which
an individual of type~$i$ has an exponentially distributed life span with mean $1/a_i$.
When it dies, it gives rise to a random number of offspring; $M$ denotes the net mean
replacement matrix, and $V\uii$ the second moment matrix of the net offspring of a type~$i$
individual, whose elements are assumed to be finite.  The matrix $AM$, where
$A := {\rm diag\,}(a_1,\ldots,a_d)$, is assumed to be irreducible, having
a simple largest eigenvalue~$\b_0 > 0$, with associated left and right eigenvectors
$\bu^T$ and~$\bv$, normalized so that $\bu^T\bone = \bu^T\bv = 1$. We write $\CCC  := \b_0 I - AM$.
Note that, from Seneta~(2006, Theorem~2.7), there are constants
$0 < \d,c_1 < \infty$ such that, for any $x \in \reals^d$,
\eq\label{C-props}
    |x^T(e^{-t\CCC } - \bv\bu^T)| \le c_1|x|e^{-\d t};
\en
$\d$ can be taken to be any value smaller than the spectral gap of the matrix~$AM$.

Define $W(t) := Z(t) e^{-\b_0 t}$.  Then it follows as in Klebaner~(1994, (1.4)) that
\eq\label{W-martingale}
      N(t)^T \Def W(t)^T + \int_0^t W^T(x) \CCC \, dx
\en
is a square integrable vector valued martingale, with predictable quadratic variation
$$
  \langle N,N \rangle_t \Eq \int_0^t e^{-2\b_0 s} \sum_{i=1}^d  a_i Z(s)^i V\uii \,ds
$$
and second moments
\eqs
  \ex\{(N(t)-N(0))(N(t)-N(0))^T\}
    &=& \int_0^t e^{-2\b_0 s} \sum_{i=1}^d  a_i \ex Z(s)^i V\uii \,ds \\
    &=& \int_0^t e^{-\b_0 s} \sum_{i=1}^d   a_i (Z_0^T e^{-\CCC s})_i V\uii \,ds\, .
\ens
In particular, we have
   $\ex |N_j(\infty) - N_j(s)|^2 \le \s_j^2 e^{-\b_0 s}$, where
$$
   \s_j^2 \Def \b_0^{-1}|Z_0|(|\bv|+c_1)\sum_{i=1}^d a_i V_{jj}\uii,
$$
and hence 
\[
    \ex\{(\bv^T N(\infty) - \bv^T N(s))^2\} \Le e^{-\b_0 s}\Bigl\{\sum_{i=1}^d \s_i v_i \Bigr\}^2
             \Le e^{-\b_0 s}|\bv|^2 \sum_{i=1}^d \s_i^2.
\]
Thus, by Doob's inequality, and writing 
$N^*(s) := \sup_{t \ge s}|N(t) - N(s)|$ and $\s^2 := \sum_{j=1}^d \s_j^2$, we have
\eq\label{N-probs}
   \pr[N^*(s) > a\s] \Le d e^{-\b_0 s}/a^2,
\en
for any $a > 0$, 
and, since $\bv^T N(t) = \bv^T W(t) = e^{-\b_0 t} \bv^T Z(t)$ because $\CCC \bv = 0$, 
\eq\label{AB-first-claim}
  \pr\Bigl[\sup_{u \ge t}|\bv^T Z(u) e^{-\b_0 u} - W| > a e^{-\b_0 t/2}\Bigr] \Le
          \frac{2\s^2|\bv|^2}{a^2},
\en
proving the first part of~\Ref{branching-claim}.

Equation~\Ref{W-martingale} can be explicitly solved for~$W$ to give
\eqa\label{W-solution}
      W(t)^T &=& W(s)^T e^{-(t-s)\CCC } + (N(t) - N(s))^T e^{-(t-s)\CCC } \non\\
             &&\quad\mbox{}  + \int_s^t (N(t) - N(x))^T e^{-(t-x)\CCC } \CCC \, dx\, ,
\ena
for any $0\le s\le t$.
In particular, taking $s=0$, letting $t=\infty$ and using~\Ref{C-props},
the first two terms on the right hand side of~\Ref{W-solution} yield
\[
   (Z_0^T\bv)\bu^T + \{(N(\infty)-N(0))^T \bv\}\bu^T.
\]
For the last term, with $s=0$, $\bu^T \CCC  = 0$ leaves
$$
    \int_0^t (N(t) - N(x))^T (e^{-(t-x)\CCC } - \bv\bu^T)\CCC \,dx,
$$
and, again using~\Ref{C-props}, we can bound the integral by
\[
    \lim_{t\to\infty} \int_0^t |N(t) - N(x)| e^{-\d(t-x)}\,dx
          \Eq \lim_{t\to\infty} \int_0^t |N(t) - N(t-y)| e^{-\d y}\,dy \Eq 0 \quad \mbox{a.s.}
\]
by dominated convergence, because~$N^*(0) < \infty$ a.s.  Furthermore, for each~$y$, we have
$\lim_{t\to\infty}|N(t) - N(t-y)|=0$ a.s. Thus
\[
     W(\infty)^T \Eq (Z_0^T\bv)\bu^T + \{(N(\infty)-N(0))^T \bv\}\bu^T,
\]
identifying $W := \lim_{t\to\infty}\bv^T Z(t) e^{-\b_0 t} = \bv^T N(\infty)$
and $W(\infty) = W\bu$.

For general $0 \le s \le t$, again
since $\bu^T \CCC  = 0$, the last term in~\Ref{W-solution} is bounded by
$2c_1\d^{-1} N^*(s)$, and the second is bounded by $(|\bv|+c_1) N^*(s)$. For the first term,
we have
$$
        |W(s)^T e^{-(t-s)\CCC } - (W(s)^T\bv)\bu^T| \Le  c_1 e^{-\d(t-s)}|W(s)|.
$$
Hence, from~\Ref{W-solution} applied twice, we have
$$
   |W(x) - W(\infty)|  \Le  c_1 e^{-\d(x-s)}|W(s)| + 2c_2 N^*(s),\quad  x \ge s,
$$
with $c_2 := 2c_1\d^{-1} + |\bv|+c_1$; and then, applying~\Ref{W-solution} once more,
\[
   |W(s)| \Le |Z_0|(|\bv| + c_1) + c_2 N^*(0).
\]
Combining these bounds, and taking $s = t(1-\f)$ with $\f := \b_0/(\b_0+2\d)$, it follows that
\[
    \sup_{x \ge t}|W(x) - W\bu| \Le c_1 e^{-\d\f t}\{|Z_0|(|\bv| + c_1) + c_2 N^*(0)\}
       + 2c_2 N^*(t(1-\f)),
\]
for any~$t>0$.  Now, from~\Ref{N-probs}, it follows that
\[
   \pr[N^*(0)e^{-\d\f t} > e^{-\d\f t/2}\s^2] \Le d e^{-\d\f t}
\]
and that
\[
    \pr[N^*(t(1-\f)) > e^{-\d\f t/2}\s^2] \Le  d e^{-\d\f t}
\]
also.  This implies the second part of~\Ref{branching-claim}, with $\chi_2 = 2\chi_1 = \d\b_0/(\b_0+2\d)$
and for suitable choice of~$K$, for~$\d>0$ smaller than the spectral gap of the matrix~$B_0$.
\hfil\proofbox

\section{Proof of Lemma~\ref{non-degenerate}}\label{4-extra}
\setcounter{equation}{0}
From Lemma~\ref{d-lem-1}, with $\e\ui = N^{-5/12}$ and $\e\ut = N^{-1}|Z_0|$,
we have
\eqa
  \lefteqn{ |\xi_N\ut(\hht_N) - e^{B_0\hht_N}N^{-1}Z_0| } \non \\   
  &&\Le  k\uh N^{-1}|Z_0|e^{\b_0\hht_N}
   \{ N^{-5/12}\log N + N^{-1}|Z_0| e^{\b_0\hht_N}\},\label{AB-PF-approx}
\ena
provided that~$\d'$ is small enough that
\[
    \log(\d'/\bv^T Z_0) \Le \log(\d_0/|Z_0|),
\]
where~$\d_0$ is as in Lemma~\ref{d-lem-1}, so that $\hht_N \le t_0(\d_0,\e\ut)$.
Since, for each $u > 0$, $e^{B_0u}$ is a positive matrix, there is a
constant~$c = c(Z_0)$ such that $e^{B_0}Z_0 \ge c\bv$, and hence such that
\eq\label{AB-PF-lower-bound}
   e^{B_0\hht_N}N^{-1}Z_0 \ \ge\ c(Z_0) N^{-1}e^{\b_0(\hht_N-1)}\bv.
\en
Then, since
\eq\label{AB-exponential-bound}
     e^{\b_0\hht_N} \Eq \frac{N\d'}{\bv^T Z_0}, 
\en
it follows that 
\[
   e^{B_0\hht_N}N^{-1}Z_0 \ \ge\ c(Z_0) e^{-\b_0}\frac{\d'}{\bv^TZ_0}\bv
\]
has all its components bounded away from zero, uniformly in~$N$.
If we now show that the right hand side of~\Ref{AB-PF-approx} is less
than half of  $c(Z_0) N^{-1}e^{\b_0(\hht_N-1)}\min_{d_1 < i\le d}\bv_i$, it will follow
from~\Ref{AB-PF-lower-bound} that the same is true of~$\xi_N\ut(\hht_N)$.
This in turn is true, provided that
\[
    \half c(Z_0) e^{-\b_0}\min_{d_1 < i\le d}\bv_i \ \ge\
      k\uh |Z_0| \{ N^{-5/12}\log N + N^{-1}|Z_0| e^{\b_0\hht_N}\}.
\]
The first factor in braces is as small as required for all~$N$ large enough, and the second,
from~\Ref{AB-exponential-bound}, is 
\[
    \frac{|Z_0|\d'}{\bv^T Z_0} \Le c'\d',
\]
for some constant~$c'$, and so can be made smaller than
\[
    \frac{c(Z_0) e^{-\b_0}\min_{d_1 < i\le d}\bv_i}{4k\uh |Z_0|}
\]
by appropriate choice of~$0 < \d' \le \d_1$.  Hence~$\xi_N\ut(\hht_N)$ has all its components
bounded uniformly away from zero, for all~$N$ sufficiently large.

It remains to consider the components of~$\xi_N\ui(\hht_N)$.  Here, Lemma~\ref{d-lem-1},
again with $\e\ui = N^{-5/12}$ and $\e\ut = N^{-1}|Z_0|$,
shows that
\[
    |\xi_N\ui(\hht_N) - x_0\ui| \Le k\ui\{N^{-5/12} + N^{-1}|Z_0| e^{\b_0\hht_N}\},
\]
and the argument above ensures that the right hand side can be made uniformly
smaller than $\half\min_{1 \le i \le d_1}x_{0i}\ui$  for all~$N$ sufficiently large,
by appropriate choice of~$0 < \d' \le \d_1$.
\hfill\proofbox

\section{Proofs of lemmas from Section~\ref{deterministic-section}.}\label{d-lemmas}
\setcounter{equation}{0}
In the proofs, we often make use of the following elementary lemma.

\begin{lem}\label{AB-little-lemma}
 Let
$(A_i(t),\,t\ge0)$, $1\le i\le i_0$,
be non-increasing processes such that, for all~$i$, $A_i(0+) \le \DDD$ and $A_i(t+) - A_i(t-) \le \DDD$
for all $t > 0$.  Suppose that, for some $a,b>0$ with $a -  b > \DDD$ and for all $0 \le t \le t_0$,
\[
     \bigcap_{i=1}^{i_0} \{A_i(t) \le a\} \ \subset\ \bigcap_{i=1}^{i_0} \{A_i(t) \le b\}\,.
\]
Then $\{A_i(t_0) \le a\}$ for all $1\le i\le i_0$.
\end{lem}

\begin{proof}
 Let $\t := \inf\{t\ge0\colon\,A_i(t) > a \ \mbox{for some}\ 1\le i\le i_0\}$.  Then, if $\t < \infty$,
there is some~$i'$, $1\le i'\le i_0$, for which $A_{i'}(\t+) \ge a$; and $A_{i}(\t-) \le a$
for all $1\le i\le i_0$ also.   Now, if $\t \le t_0$, the latter observation implies that
$A_{i}(\t-) \le b$ for all $1\le i\le i_0$, and in particular that $A_{i'}(\t-) \le b$.
Thus $A_{i'}(\t+) - A_{i'}(\t-) \ge a-b > \DDD$, a contradiction.  Hence $\t > t_0$,
proving the lemma.
\end{proof}

\bigskip
\nin{\bf Proof of Lemma~\ref{d-lem-1}.}\ 
Define
\eqs
           \D\ui(t) &:=& \sup_{0\le u\le t}|\xi\ui(u) - x_0\ui|; \qquad
           \D\ut(t) \Def \sup_{0\le u\le t} e^{-\b_0 u}|\xi\ut(u)|; \\
    \D\uh(t) &:=& \sup_{0\le u\le t} e^{-\b_0 u}|\xi\ut(u) - e^{B_0u}\xi\ut(0)|.
\ens
Then it follows directly from~\Ref{4-d} that, if~$t$ is such that
\eq\label{cond-1-d}
    \max\{\D\ui(t), e^{\b_0t}\D\ut(t)\} \Le \r_2/2,
\en
then
\eqs
   \lefteqn{e^{-\b_0 t}|\xi\ut(t)|}\\
    &\le& \cst_2\Blb\e\ut + \int_0^t e^{-\b_0u} \|DB\|_{\r_2}
     \{|\xi\ut(u)|^2 + |\xi\ui(u) - x_0\ui|\,|\xi\ut(u)|\}\,du \Brb,
\ens
from which it follows that
\eq\label{inequality-1-d}
  \D\ut(t) \Le \cst_2\Bigl\{\e\ut + \b_0^{-1}\|DB\|_{\r_2}e^{\b_0t}\{\D\ut(t)\}^2
             + \|DB\|_{\r_2} \D\ut(t)\int_0^t \D\ui(u)\,du\Bigr\}.
\en
Then, from~\Ref{3a-d},
\eqs
|\xi\ui(t) - x_0\ui| 
   &\le& \cst_1\Blb\e\ui
     + \int_0^t \{ \|A\|_{\r_2}|\xi\ut(u)| + \kkkc  |\xi\ui(u) - x_0\ui|^2\}\,du
          \Brb,
\ens
giving
\eq\label{inequality-2-d}
  \D\ui(t) \Le \cst_1\{\e\ui + \b_0^{-1}\|A\|_{\r_2} \D\ut(t) e^{\b_0t}
          + \kkkc  \D\ui(t) \int_0^t \D\ui(u)\,du\}.
\en
Hence, if~$t$ is also such that
\eq\label{cond-2-d}
     \cst_1 \kkkc  \int_0^t \D\ui(u)\,du \Le 1/2, 
\en
it follows from \Ref{inequality-1-d} and~\Ref{inequality-2-d} that
\eq\label{shorter-1-d}
   \D\ui(t) \Le 2\cst_1\{\e\ui + \b_0^{-1}\|A\|_{\r_2} \D\ut(t) e^{\b_0t}\}.
\en

Now, from~\Ref{shorter-1-d},
\[
   \cst_2 \|DB\|_{\r_2} \int_0^t \D\ui(u)\,du \Le 2\cst_2 \|DB\|_{\r_2} \cst_1\Blb \e\ui t
    + \frac{\|A\|_{\r_2}\D\ut(t)}{\b_0^2}\,e^{\b_0t} \Brb \Le \frac13,
\]
if
\eq\label{new-cond-1-d}
  \e\ui t \Le \frac1{12\cst_2\cst_1 \|DB\|_{\r_2}}\quad\mbox{and}\quad
     \D\ut(t)e^{\b_0t} \Le \frac{\b_0^2}{12\cst_2\cst_1 \|DB\|_{\r_2}\|A\|_{\r_2}},
\en
so that, if also
\eq\label{new-cond-2-d}
    \D\ut(t)e^{\b_0t} \Le \frac{\b_0}{6\cst_2 \|DB\|_{\r_2}},
\en
it follows from~\Ref{inequality-1-d} that
\eq\label{shorter-2-d}
     \D\ut(t) \Le 2\cst_2\e\ut  .
\en

Thus, for  $t \le t_0(\d_0,\e\ut)$, assuming the bounds \Ref{cond-1-d}, \Ref{cond-2-d},
\Ref{new-cond-1-d} and~\Ref{new-cond-2-d} implies that the bounds \Ref{shorter-1-d} and~\Ref{shorter-2-d}
also hold.  We now show that these, if~$0 < \d_0 \le 1$ is chosen to be suitably small, in turn imply
that strictly smaller bounds are valid in the right hands
sides of \Ref{cond-1-d}, \Ref{cond-2-d}, \Ref{new-cond-1-d} and~\Ref{new-cond-2-d}.
For $t \le t_0(\d_0,\e\ut)$,  \Ref{shorter-2-d} implies that $\D\ut(u)e^{\b_0u} \le 2\cst_2\d_0$
for all $0\le u\le t$, so that \Ref{new-cond-1-d} and~\Ref{new-cond-2-d} are satisfied
with the right hand sides halved if
\eq\label{new-cross-2}
   2\cst_2\d_0 \Le \frac12\min\Blb \frac{\b_0}{6\cst_2\|DB\|_{\r_2}},
           \frac{\b_0^2}{12\cst_2\cst_1 \|DB\|_{\r_2}\|A\|_{\r_2}}\Brb ,
\en
because, by Assumption~\Ref{e-e-bnd},
\eq\label{new-cross-3}
    \e\ui \log (1/\e\ut) \Le \frac{\b_0}{24\cst_2\cst_1 \|DB\|_{\r_2}}.
\en
The same is true for \Ref{cond-1-d} and~\Ref{cond-2-d} if  $2\cst_2\d_0 \le \r_2/4$ and
\eqa
  &&2\cst_1\{\e\ui + 2\cst_2\d_0 \b_0^{-1}\|A\|_{\r_2}\} \Le \r_2/4; \label{new-cross-1}\\
  &&2\cst_1^2\kkkc \b_0^{-1}\{\e\ui \log(1/\e\ut) + 2\cst_2\d_0 \b_0^{-1}\|A\|_{\r_2}\} \Le 1/4.
              \label{new-cross-1.5}
\ena
The contributions from $\e\ui$ in~\Ref{new-cross-1} and from $\e\ui \log(1/\e\ut)$ in~\Ref{new-cross-1.5}
are less that half their right hand sides, because of Assumptions \Ref{e-i-bnd} and~\Ref{e-e-bnd},
and the remaining terms can be made small enough by choice of~$\d_0$.  Then, since we can
take $\DDD=0$ in Lemma~\ref{AB-little-lemma}, it follows that \Ref{shorter-1-d} and~\Ref{shorter-2-d}
hold for all $0\le t\le t_0(\d_0,\e\ut)$, if~$\d_0 > 0$ is chosen small enough.

For the third statement, we use~\Ref{4-d} to give
\eqs
  \lefteqn{e^{-\b_0t} |\xi\ut(t) - e^{B_0t}\xi\ut(0)| } \\
    &\le& \cst_2\|DB\|_{\r_2} \int_0^u
         \{\D\ui(u)  + e^{\b_0u}\D\ut(u)\} \D\ut(u)\,du.
\ens
Then, for $t \le t_0(\d_0,\e\ut)$, it follows from \Ref{shorter-1-d} and~\Ref{shorter-2-d},
for suitable constants~$k_i$, 
that
\eqa
   \D\uh(t) &\le& k_8 \int_0^u
         \{\e\ui +  e^{\b_0u}\e\ut \} \e\ut\,du  
   \Le k_9\{\e\ui\e\ut t + (\e\ut)^2 e^{\b_0t}\},\phantom{XXX}
         \label{D3-d}
\ena
establishing the third statement of the lemma.

The proof of the remaining inequalities 
is somewhat similar. Define
\[
   \tD\ui(t) \Def \sup_{0\le u\le t}|\xi\ui(u) - \txi\ui(u)|;\quad\!\! 
   \tD\ut(t) \Def \sup_{0\le u\le t}\{e^{-\b_0 u}|\xi\ut(u) - \txi\ut(u)|\}.
\]
First, from~\Ref{3a-d},
\eqa
  |\xi\ui(t) - \txi\ui(t)|
     &\le&  \int_0^t |e^{C(t-u)} \{A(\xi(u))\xi\ut(u) - A(\txi(u))\txi\ut(u)\}|\,du
                 \non \\
   &&\qquad\mbox{}   + \int_0^t |e^{C(t-u)}\{\tc(\xi\ui(u)) - \tc(\txi\ui(u))\} |\,du \non\\
   &&\qquad\qquad\mbox{}            + |e^{Ct}(\xi\ui(0) - \txi\ui(0))|.\label{1-compt-d}
\ena
The final term of~\Ref{1-compt-d} is no bigger than $\cst_1 k\uv\e\ut\log(1/\e\ut)$, and, for~$t$ such that
\eq\label{delta-conds-d-1}
    \cst_1 \kkkc  \int_0^t \tD\ui(u)\,du \Le 1/8\quad\mbox{and}\quad
    \cst_1 \kkkc  \int_0^t \D_1(u)\,du \Le 1/8, 
\en
the middle term of~\Ref{1-compt-d} is bounded by $\tfrac14\tD\ui(t)$.
Note that, from \Ref{shorter-1-d} and~\Ref{shorter-2-d},
\[
    \cst_1 \kkkc  \int_0^t \D_1(u)\,du \Le 2\kkkc \cst_1^2\b_0^{-1}\e\ui\log(1/\e\ut)
                + 4\kkkc \cst_1^2\cst_2\|A\|_{\r_2}\d
\]
for $t \le t_0(\d,\e\ut)$ for any $\d \le \d_0$, and the right hand side is less than~$1/8$
if~$\d$ is chosen small enough, by Assumption~\Ref{e-e-bnd}.
The first term in~\Ref{1-compt-d} we can bound by
\eq\label{tilde-first}
   \cst_1 \int_0^t  \Bigl\{\|DA\|_{\r_2}\{\tD\ui(u) + e^{\b_0u}\tD\ut(u)\}|\xi\ut(u)|
        + \|A\|_{\r_2} e^{\b_0u}\tD\ut(u)\Bigr\}\,du.
\en
Here, the final element is at most $\cst_1  \|A\|_{\r_2} e^{\b_0t}\tD\ut(t)/\b_0$;
since $|\xi\ut(u)| \le 2\cst_2\e\ut e^{\b_0u}$ from~\Ref{shorter-2-d}, the remaining elements
are bounded by
\[
   \frac{\cst_1\|DA\|_{\r_2}}{\b_0}\,\{e^{\b_0t}\tD\ut(t) + \tD\ui(t)\}\,2\cst_2\d_0
\]
if also $t \le t_0(\d_0,\e\ut)$.  Choosing
\[
   \d_0' \Def \min\Bigl\{ \d_0, \frac{\b_0}{8\cst_2\cst_1\|DA\|_{\r_2}}\Bigr\} ,
\]
it thus follows from~\Ref{1-compt-d} that, for~$t$ such that~\Ref{delta-conds-d-1} is satisfied,
\eq
  \tD\ui(t) \Le \frac{2\cst_1}{\b_0} e^{\b_0 t}\tD\ut(t)\{\|A\|_{\r_2} + 2\cst_2\d_0'\|DA\|_{\r_2} \}
                +  2\cst_1  k\uv\e\ut\log(1/\e\ut). \label{9-det-*-d}
\en

For the second components, \Ref{4-d} gives
\eqa
  \lefteqn{|\xi\ut(t) - \txi\ut(t)|
    \Le   \int_0^t |e^{B_0(t-u)}(B(\xi(u)) - B(\txi(u)))\txi\ut(u)|\,du} \phantom{XXXX} \non\\
   &&\mbox{}\qquad\qquad + \int_0^t |e^{B_0(t-u)}(B(\xi(u)) - B_0)(\xi\ut(u) - \txi\ut(u))|\,du \non\\
    &&\quad\ \qquad\qquad\mbox{}\  +  |e^{B_0 t}(\xi\ut(0) - \txi\ut(0))|.
\ena
If
\eq\label{delta-conds-d}
    \tD\ut(t) \Le k^{(4)}\e\ut,
\en
$|\txi\ut(u)| \le k_{10}\e\ut e^{\b_0 u}$ for some~$k_{10}$,
and so, for any $\d \le \d_0'$ and all $t \le t_0(\d,\e\ut)$ and using \Ref{shorter-1-d}
and~\Ref{shorter-2-d}, we have
\eqs
   \tD\ut(t) &\le&   k_{11}\Bigl\{(\d + \e\ui)\int_0^t \tD\ut(u)\,du
      + \e\ut \int_0^t \tD\ui(u)\,du\Bigr\}  + \cst_2\e\uh.
\ens
It then follows from~\Ref{9-det-*-d} for $t \le t_0(\d,\e\ut)$ satisfying~\Ref{delta-conds-d} that
\[
   \tD\ut(t) \Le   k_{12}\Blb(\d + \e\ui)\int_0^t \tD\ut(u)\,du  + (\e\ut)^2 \{\log(1/\e\ut)\}^2 \Brb
               + \cst_2\e\uh,
\]
and hence, by Gronwall's inequality and the restrictions on $\e\uh$, that
\eq\label{9-2-det-d}
    \tD\ut(t) \Le k_{13} (\e\ut)^{1+\g}\exp\{k_{12}t(\d + \e\ui)\}.
\en
Now, taking any $\g' < \g$ and then choosing~$\d_1(\g') := \min\{ \d_0',\g'\b_0/k_{12}\}$,
this gives a bound for~$\tD\ut(t_0(\d_1(\g'),\e\ut))$ of order
$O((\e\ut)^{1+\g-\g'})$, because of~\Ref{e-e-bnd}.
Taking $\g' = \g/2$ and $\d_1 := \d_1(\g/2)$, we now conclude that, for all $t \le t_0(\d_1,\e\ut)$,
\eq\label{AB-2.1-last-1}
   \sup_{0\le u\le t}\{e^{-\b_0 u}|\xi\ut(u) - \txi\ut(u)|\}
           \Le k_{14}(\e\ut)^{1 + \g/2},
\en
and that, using~\Ref{9-det-*-d},
\eq\label{AB-2.1-last-2}
\sup_{0\le u\le t_0(\d,\e\ut)}|\xi\ui(u) - \txi\ui(u)|  \Le k_{15}(\e\ut)^{\g/2}.
\en
Comparing these bounds with those in \Ref{delta-conds-d-1} and~\Ref{delta-conds-d}
and applying Lemma~\ref{AB-little-lemma} with $\DDD = 0$ now shows that
\Ref{AB-2.1-last-1} and~\Ref{AB-2.1-last-2}
are satisfied for all $t \le t_0(\d_1,\e\ut)$ and~$\e\ut \le \min\{k^{(8)},\d_1\}$,
for some~$k^{(8)}$ small enough.
\hfil\proofbox\par\smallskip\par

\medskip

\bigskip
\nin{\bf Proof of Lemma~\ref{d-lem-2}.}\ 
Define
\eqs
           \D_\d\ui(t) &:=& \sup_{0\le u\le t} e^{\k'u}|\xi_\d\ui(u) - x_0\ui|; \qquad
           \D_\d\ut(t) \Def \sup_{0\le u\le t} e^{\b_1 u}|\xi_\d\ut(u)|; \\
    \D_\d\uh(t) &:=& \sup_{0\le u\le t} e^{\b_1 u}|\xi_\d\ut(u) - e^{B_0u}x_{\d0}|.
\ens
Then it follows directly from~\Ref{4-d} that, if~$t$ is such that
\eq\label{cond-1*-d}
    \max\{\D_\d\ui(t), e^{-\b_1t}\D_\d\ut(t)\} \Le \r_2/2,
\en
then
\eqs
   e^{\b_1 t}|\xi_\d\ut(t)| &\le& \cst_2'\Blb \d + \int_0^t e^{\b_1u} \|DB\|_{\r_2}
     \{|\xi_\d\ut(u)|^2 + |\xi_\d\ui(u) - x_0\ui|\,|\xi_\d\ut(u)|\}\,du \Brb,
\ens
where $\cst_2'$ is as defined following~\Ref{part-2-diff*}. It thus follows that
\eq\label{inequality-1*-d}
  \D_\d\ut(t) \Le \cst_2'\Bigl(\d + \|DB\|_{\r_2}\D_\d\ut(t)\{\b_1^{-1}\D_\d\ut(t) + (\k')^{-1}\D_\d\ui(t)\}
     \Bigr).
\en
Then, from~\Ref{3a-d},
\eqs
   |\xi_\d\ui(t) - x_0\ui| &\le& \cst_1\Bigl\{ |x_{\d0}\ui - x_0\ui|e^{-\k t} \\
    &&\mbox{}\ + \int_0^t e^{-\k(t-u)}\{ \|A\|_{\r_2}|\xi_\d\ut(u)| + \kkkc  |\xi_\d\ui(u) - x_0\ui|^2\}\,du \Bigr\},
\ens
and, because $be^{-bt}\int_0^t e^{au}\,du \le 1$ for all $b \ge 0$ and $-\infty < a \le b$,
this gives
\eq\label{inequality-2*-d}
  \D_\d\ui(t) \Le \cst_1\bigl\{ \d + (\k-\k')^{-1}\|A\|_{\r_2} \D_\d\ut(t)
          + (\k- \k')^{-1}\kkkc  (\D_\d\ui(t))^2 \bigr\}.
\en
Hence, if~$t$ is also such that
\eq\label{cond-2*-d}
     \D_\d\ui(t) \Le \min\Blb \frac{\k-\k'}{2\cst_1 \kkkc },\, \frac{\k'}{4\cst_2'\|DB\|_{\r_2}} \Brb;
          \quad \D_\d\ut(t) \le \frac{\b_1}{4\cst_2'\|DB\|_{\r_2}},
\en
it follows from \Ref{inequality-1*-d} and~\Ref{inequality-2*-d} that
\eqa
     \D_\d\ut(t) &\le& 2\cst_2'\d ;          \label{shorter-2*-d} \\
    \D_\d\ui(t) &\le& 2\cst_1\{\d + (\k-\k')^{-1}\|A\|_{\r_2} \D_\d\ut(t)\} \non\\
                      &\le& 2\cst_1\d\{1 + 2\cst_2'(\k-\k')^{-1}\|A\|_{\r_2}\}.
                       \label{shorter-1*-d}
\ena
Comparing these bounds with \Ref{cond-1*-d} and~\Ref{cond-2*-d}
and invoking Lemma~\ref{AB-little-lemma} with $\DDD = 0$ shows that \Ref{shorter-2*-d}
and~\Ref{shorter-1*-d} hold for all~$t$,
if~$\d$ is chosen sufficiently small, proving the first two statements of the lemma.
For the third, bound the difference $|\xi_\d\ut(u) - e^{B_0u}x_{\d0}\ut|$ using~\Ref{4-d}.

For the final statements, we define
\[
   \tD_\d\ui(t) \Def \sup_{0\le u\le t}e^{\k'u}|\xi_\d\ui(u) - \txi_\d\ui(u)|;\quad 
   \tD_\d\ut(t) \Def \sup_{0\le u\le t}\{e^{\b_1 u}|\xi_\d\ut(u) - \txi_\d\ut(u)|\},
\]
and argue much as for Lemma~\ref{phase-2-lem-new*}, with $x_{N,\d}$ replaced by~$\txi_\d$.
%
To start with, using the counterpart of~\Ref{expanded-bnd} with $\h_{N1}'$ replaced by $e^{-\k t}\e\uf$,
we obtain, in analogy to~\Ref{9*},
\eq\label{tdelta-1}
    \tD_\d\ui(t) \Le k'_{8}\{\tD_\d\ut(t) + \e\uf e^{(\k'-\k)t}\},
\en
for any $\d \le \d_1$, where~$\d_1$ is as chosen in the proof of Lemma~\ref{phase-2-lem-new*},
and for $t$ such that
\eq\label{conds-5}
    \max\{\tD_\d\ui(t), \tD_\d\ut(t)\} \Le \d_1.
\en
Then we use the analogue of~\Ref{part-2-diff*}, with the term in moduli replaced by~$e^{-\b_1t}\e\uf$,
and~\Ref{shorter-2*-d} to obtain, in analogy to~\Ref{previous-1*},
\[
    \tD_\d\ut(t) \Le k'_{9}\d \int_0^t (\tD_\d\ut(u) + \tD_\d\ui(u))\,du  + \e\uf,
\]
for $\d \le \d_1$. These are combined, as for \Ref{dN-1-este*} and~\Ref{GR}, to yield
\eq\label{tdelta-2}
   \tD_\d\ut(t) \Le k'_{10}\e\uf \exp\{k'_{11}t\d \},
\en
and this yields $\tD_\d\ut(t_1(\d,\h)) \Le k'_{12}\e\uf\h^{-\th}$, if
$\d = \d(\th) := \min\{ \d_1, \b_1\th/k'_{11}\}$.   The corresponding bound
\eq\label{tdelta-1a}
    \tD_\d\ui(t_1(\d,\h)) \Le k'_{13}\e\uf\h^{-\th}
\en
now follows from~\Ref{tdelta-1},
and these imply that the bound in~\Ref{conds-5}
can be replaced by~$\d_1/2$ for $0 \le t \le t_1(\d(\th),\h)$, provided that
\eq\label{eps-restn}
     \e\uf\h^{-\th} \Le \KKK \Def \d_1/2\max\{k'_{12},k'_{13}\}.
\en
Invoking Lemma~\ref{AB-little-lemma} with $\DDD = 0$ completes the proof of Lemma~\ref{d-lem-2}.
\hfil\proofbox\par\smallskip\par

\section{Complements}\label{complements}
\setcounter{equation}{0}

\medskip
\nin{\bf Breakdown of the branching approximation.}

\medskip\nin
To illustrate that one cannot take an exponent~$\a \le 1/3$ in Lemma~\ref{phase-1-lem},
consider the stochastic logistic growth model~$X_N$ given by the transition rates
\[
    X \ \to\ X + 1  \quad\mbox{at rate}\ X(1-X/N),\qquad X \in \{0,1,\ldots,N\},
\]
with deterministic equations $\dot \xi = \xi(1-\xi)$ and unstable equilibrium $\bar x = 0$.
The corresponding Markov branching process~$Z$ has transition rates
\[
   X \ \to\ X + 1  \quad\mbox{at rate}\ X,\qquad X \in \Z_+.
\]
The likelihood ratio of the distribution of~$X_N$ with respect to that of~$Z$ at a path
segment of~$m$ jumps, starting from~$1$, with waiting times~$T_i$ in states $1\le i \le m$,
is given by
\[
    R_m(T_1,\ldots,T_m) \Def \prod_{i=1}^m e^{i/N}(1-i/N)\exp\{N^{-1}(i^2 T_i - i)\}.
\]
The product $\prod_{i=1}^m e^{i/N}(1-i/N)$ converges as $m\to\infty$ to a limit smaller than~$1$.
Then, under the distribution $\bP^Z$ of~$Z$, $T_i \sim i^{-1}E_i$, $i\ge1$, where
$(E_i,\,i\ge1)$ are independent standard
exponential random variables. It thus follows from Lyapounov's central limit theorem that, for any $a > 0$ and
$m = m(N) = \lfloor aN^{2/3}\rfloor$,
\[
    N^{-1}\sum_{i=1}^{m(N)} N^{-1}(i^2 T_i - i) \ \to_d\ \nn(0,a^3/3)
\]
under $\bP^Z$, as $N\to\infty$.  Hence
\[
    -\ex^Z \min\{0, (R_{m(N)}(T_1,\ldots,T_m) - 1)\}\ \to\ c_a \ >\ 0
\]
as $N \to \infty$, so that, from~\Ref{AB-dtv-def},
for path segments of length at least $aN^{2/3}$, the total variation
distance between the distributions of $X_N$ and~$Z$ is not asymptotically negligible.
\hfil\proofbox

\newpage
\nin{\bf There is no universal exponent~$\g$ possible in~\Ref{deterministic-approximation-final}.}

\medskip\nin
We consider an example with $d_1=0$ and $d_2=2$, having
$$
      g^{(1,0)}(x) \Def (1-\h)x_1 + \h x_2 \quad\mbox{and}\quad g^{(0,1)}(x) \Def \h x_1 + (1-\h)x_2,
$$
and initial state $x_{N,0} = x_{N,0}\ut = N^{-1}(1,1)^T$; we take any $0 < \h < 1/2$.  This yields
the deterministic solution $\xi_N(t) = N^{-1}e^t (1,1)^T$.  Then
$$
    B_0^T \Eq \Bl \begin{array}{cc} 1-\h & \h \\ \h & 1-\h \end{array} \Br,
$$
its eigenvalues
are $\b_0 = 1$ and $\b' = 1-2\h$, and the left and right eigenvectors associated with
the eigenvalue~$1$ are $\bu^T = \half(1,1)$ and $\bv = (1,1)^T$.
The process~$Nx_N\ut$ coincides with the branching process~$Z$, which has two square
integrable martingales
\[
    W^\bv(t) \Def (Z_1(t) + Z_2(t))e^{-t} \quad\mbox{and}\quad W'(t) \Def (Z_1(t) - Z_2(t))e^{-t(1-2\h)},
\]
with a.s.\ limits $W$ and~$W'$ respectively; note that $W>0$ a.s., and that $W'$ is non-degenerate,
and has a distribution symmetric about zero.  Thus, recalling~\Ref{AB-tau-star-defs},
asymptotically as $N\to\infty$,
\eqs
    \lefteqn{x_{N1}(\t_{N*}^x + (5/12)\log N) - x_{N2}(\t_{N*}^x + (5/12)\log N) } \\
 &&\sim\    N^{-1}(N/W)^{1-2\h}\,W'(\log (N/W)) \ \sim\ N^{-2\h} \,\{W'/W^{1-2\h}\}\,,
\ens
whereas
\[
    \xi_{N1}(\t_{N*}^\xi + (5/12)\log N) - \xi_{N2}(t_{N*}^\xi + (5/12)\log N) \Eq 0.
\]
Thus, in~\Ref{deterministic-approximation-final}, we would have to have $\g \le 2\h$,
and since we can take any $0 < \h < 1/2$ in the example, there is no universal choice
of~$\g$ possible.

\bigskip\bigskip

\end{document}